\documentclass[reqno,english]{amsart}
\usepackage{amsfonts,amsmath,latexsym,verbatim,amscd,mathrsfs,color,array}

\usepackage{amsmath,amssymb,amsthm,amsfonts,graphicx,color}
\usepackage{amssymb}
\usepackage{pdfsync}
\usepackage{epstopdf}
\usepackage[colorlinks=true]{hyperref}
\usepackage{subfigure}

\makeatletter
\def\@currentlabel{2.1}\label{e:dispaa}
\def\@currentlabel{2.21}\label{e:dispau}
\def\@currentlabel{2.22}\label{e:dispav}
\def\@currentlabel{2.23}\label{e:dispaw}
\def\@currentlabel{2.24}\label{e:dispax}
\def\theequation{\thesection.\@arabic\c@equation}
\makeatother
\let\oldbibliography\thebibliography
\renewcommand{\thebibliography}[1]{%
\oldbibliography{#1}%
\setlength{\itemsep}{0pt}%
}

\renewcommand{\theequation}{\thesection.\arabic{equation}}
\newtheorem{lemma}{Lemma}[section]

\newtheorem{proposition}{Proposition}[section]
\newtheorem{corollary}{Corollary}[section]
\newtheorem{remark}{Remark}[section]
\newcommand{\bremark}{\begin{remark} \em}
\newcommand{\eremark}{\end{remark} }

\newtheorem*{thmA}{Theorem A}
\newtheorem*{thmB}{Theorem B}
\newtheorem*{thmC}{Theorem C}
\newtheorem*{thmD}{Theorem D}
\newtheorem{open problem}{Open Problem}[section]
\newtheorem{open question}{Open Question}[section]
\newtheorem{open questions on hexagonal lattice}{Open Questions on hexagonal lattice}[section]
\newtheorem{theorem}{Theorem}[section]

\newtheorem*{problemA}{Problem A}
\newtheorem*{problemB}{Problem B}
\newtheorem*{problemC}{Problem C}

\newcommand{\R}{{\mathbb R}}

\newcommand{\BE}{\begin{equation}}
\newcommand{\BEN}{\begin{equation*}}
\newcommand{\EE}{\end{equation}}
\newcommand{\EEN}{\end{equation*}}
\newcommand{\BL}{\begin{lemma}}
\newcommand{\EL}{\end{lemma}}
\newcommand{\BT}{\begin{theorem}}
\newcommand{\ET}{\end{theorem}}
\newcommand{\BP}{\begin{proposition}}
\newcommand{\EP}{\end{proposition}}
\newcommand{\BC}{\begin{corollary}}
\newcommand{\EC}{\end{corollary}}
\renewcommand{\Re}{\operatorname{Re}}
\renewcommand{\Im}{\operatorname{Im}}
\usepackage{amsmath}
\DeclareMathOperator*{\argmax}{arg\,max}
\DeclareMathOperator*{\argmin}{arg\,min}

\begin{document}

%\title[Green Function]{Two-component reduced functional}
%%%%%%%%%%%%%%%%%%%%%%%%%%%%%%%%%%%%%%%%%%%%%%%%%%%%%%%%%%%%%%%%%%%%%%

\title[Ratio of theta functions]{\bf On ratios of theta functions}

\author{Senping Luo}

\author{Juncheng Wei}

\address[S.~Luo]{School of Mathematics and statistics, Jiangxi Normal University, Nanchang, 330022, China.}
\address[J.~Wei]{Department of Mathematics, Chinese University of Hong
Kong, Shatin, NT, Hong Kong.}

\email[S.~Luo]{luosp1989@163.com}

\email[J.~Wei]{wei@math.cuhk.edu.hk}

\begin{abstract} Motivated by the average partition function of c free bosons $($Afhkami-Jeddi et al. \cite{Afhk2021}$)$ and the average of the genus 1 partition function over the Narain moduli space $($Maloney-Witten \cite{Witten2020}$)$, we investigate ratios of theta functions.
In this paper, we completely classify the minimizers (or maximizers) for
ratios of theta and Epstein zeta functions. We find that the hexagonal lattice plays a pivotal role there.
These results have direct applications in conformal and Liouville field theory via partition functions. Additionally, they yield the minima of differences of theta and Epstein zeta functions, which have implications for the mathematics of crystallization and interacting particle theory (\cite{Bet2016,Bet2019AMP}).

\end{abstract}

\maketitle

\setcounter{equation}{0}

\section{Introduction and main results}
\setcounter{equation}{0}

Let $ z\in \mathbb{H}:=\{z\in\mathbb{C}: \Im(z)>0\}$. The Epstein zeta and theta functions associated with the lattice $ \Lambda$ are defined as
 \begin{equation}\aligned
 \nonumber
\zeta(s,\Lambda):=\sum_{\mathbb{P}\in \Lambda\setminus\{0\}}\frac{1}{|\mathbb{P}|^{2s}},\;\;
\theta(\alpha,\Lambda):=\sum_{\mathbb{P}\in\Lambda} e^{-\pi\alpha |\mathbb{P}|^2}.
\endaligned\end{equation}
By the parametrization $\Lambda =\sqrt{\frac{1}{\Im(z)}}\Big({\mathbb Z}\oplus z{\mathbb Z}\Big)$, one has
 \begin{equation}\aligned
 \label{thetas}
&\zeta(s,z):=\zeta(s,\Lambda)=\sum_{(m,n)\in\mathbb{Z}^2\backslash\{0\}}\frac{\Im(z)^s}{|mz+n|^{2s}},\;\\
&\theta(\alpha,z):=\theta(\alpha,\Lambda)=\sum_{(m,n)\in\mathbb{Z}^2}e^{-\pi \alpha\cdot\frac{|mz+n|^{2}}{\Im(z)}}.
\endaligned\end{equation}

In 1950s, in a series of work in number theory, Rankin \cite{Ran1953}, Cassels \cite{Cas1959}, Ennola \cite{Enn1964a}, and Diamond \cite{Dia1964} established that
\begin{thmA}[Rankin, Cassels, Ennola, Diamonda 1950-1960s] For $s>1$, up to the action by the modular group,
\begin{equation}\aligned\nonumber
\argmin_{z\in\mathbb{H}}\zeta(s,z)=e^{i\frac{\pi}{3}}.
\endaligned\end{equation}
\end{thmA}
The analysis of high-dimensional Epstein zeta functions is significantly more challenging; the first rigorous result in this direction was established by Sarnak-Str\"{o}mbergsson \cite{Sarnak2006} in 2006 .
Motivated by {\bf Theorem A},
Montgomery \cite{Mon1988} further proved that
\begin{thmB}[Montgomery 1988] For $\alpha>0$, up to the action by the modular group,
\begin{equation}\aligned\nonumber
\argmin_{z\in\mathbb{H}}\theta(\alpha,z)=e^{i\frac{\pi}{3}}.
\endaligned\end{equation}
\end{thmB}

A third class of fundamental modular invariant functions, distinct from the theta and Epstein zeta functions, is given by
\begin{equation}\aligned\label{eta}
{\sqrt{\Im(\tau)}|\eta(\tau)|^2}.
\endaligned\end{equation}
Here, the Dedekind eta function $\eta(\tau)$ is defined as
\begin{equation}\aligned\nonumber
\eta(\tau)=q^{\frac{1}{24}}\prod_{n=1}^\infty (1-q^n),\;\; q=e^{2\pi i\tau}.
\endaligned\end{equation}
$\eta^{24}(\tau)$ is the discriminant function from the theory of elliptic functions (Zagier \cite{Zagier}).
In their study of extremals for the determinants of Laplacians, Osgood-Phillips-Sarnak \cite{Sarnak1988} found that
\begin{thmC}[Osgood-Phillips-Sarnak, page 206]\label{Sarnak}
Up to the action by the modular group,
 \begin{equation}\aligned\nonumber
\argmax_{\tau\in\mathbb{H}}\sqrt{\Im(\tau)}|\eta(\tau)|^2=e^{i\frac{\pi}{3}}.
\endaligned\end{equation}
\end{thmC}

Before proceeding, we discuss the interrelations among Theorems A-C.
Through the Mellin transform
$$
\zeta(s,z)=\frac{\pi^s}{\Gamma(s)}\int_0^\infty \big(\theta(\alpha,z)-1\big)\alpha^{s-1}d\alpha,
$$
Theorem B implies Theorem A. Furthermore, via the Kronecker first limit formula (see, e.g., \cite{Cohen2007}), we obtain
${\sqrt{\Im(\tau)}|\eta(\tau)|^2}$ from the Epstein zeta function as follows:
 \begin{equation}\aligned
 \zeta(s,\tau)={\pi  \over s-1}+2\pi (\gamma -\log(2)-\log({\sqrt {\Im(\tau)}}|\eta (\tau )|^{2}))+O(s-1), \;\;s\rightarrow 1^+,
\endaligned\end{equation}
where $\gamma$ is the Euler-Mascheroni constant. Consequently, Theorem A yields Theorem C. On the other hand, the Dedekind eta function can be expressed in terms of a difference of theta functions (Nakayama \cite{Naka2004}, page 233). Specifically,
\begin{equation}\aligned\label{Prop1}
\sqrt{\Im(\tau)}|\eta(\tau)|^2=-\frac{\sqrt6}{4}\Big(\theta(\frac{3}{2},\tau)-2\theta(6,\tau)\Big).
\endaligned\end{equation}
Sandier-Serfaty \cite{Serfaty2012} and Chen-Oshita \cite{Che2007} reduce their respective functionals to the Dedekind eta function in distinct problems, thereby providing independent and completely different proofs of Theorem C. Another proof is given in Nonnenmacher and Voros \cite{Non1998}. Alternatively, a direct proof of Theorem C can be derived using \eqref{Prop1} and Theorem 1 in \cite{LW2023}.

The celebrated Theorems A-C have profound applications across various fields. These theorems lay the foundation for the optimality of triangular (hexagonal) vortices in Ginzburg-Landau theory (Abrikosov \cite{Abr}, Sandier-Serfaty \cite{Serfaty2012,Serfaty2018}).
Theorems A-C have direct applications to crystallization among lattices (B\'etermin \cite{Bet2016}), Ohta-Kawasaki models in di-block copolymers (Chen-Oshita \cite{Che2007}, Goldman-Muratov-Serfaty \cite{GMS2013}), Bose-Einstein condensates (\cite{Ho2001}), and the crystallization of particle interactions (B\'etermin \cite{Bet2019,Betermin2021JPA,Betermin2021AHP,Betermin2021a}, Blanc-Lewin \cite{Bla2015}, Luo-Ren-Wei \cite{Luo2019}), among others.
Furthermore, theta functions  are deeply connected to string theory (Alvarez-Gaum\'e-Moore-Vafa \cite{Vafa1986}),
 Gauss core model  (Cohn and Courcy-Ireland \cite{Cohn2018},
  Prestipino-Saija-Giaquinta \cite{Pres2011}), sphere packings (Conway-Sloane \cite{Conway},
   Viazovska \cite{Via2017}, Cohn-Kumar-Miller-Radchenko-Viazovska \cite{Coh2017}),
 the reverse Minkowski inequality (Regev and Stephens-Davidowitz \cite{Regev2024}, Regev \cite{Regev2023}), and
  communications (Barreal-Damir-Freij-Hollanti \cite{Bar2020}). A recent result on the difference of Epstein zeta functions and its application can be found in Luo-Wei (\cite{LW2025}). For the further applications and background of theta functions, we refer to \cite{Betermin2020CA,Betermin2023JAM,Betermin2021bb,Kramer1991}. Various
interesting advances in on theta related functions are made in B\'etermin \cite{Betermin2020CA,Betermin2021SIAM}, and B\'etermin-Faulhuber \cite{Betermin2023JAM}, B\'etermin-Faulhuber-Steinerberger \cite{Betermin2021bb}.
This enduring relevance is captured by Mumford's \cite{Mumford1983} observation: {\em "The theory of theta functions is far from a finished polished topic."}

Consequently, this motivates the examination of the ratio forms of Theorems A-C. Such ratios--specifically involving theta, Epstein zeta, and Dedekind eta functions--are fundamental to conformal field theory and string theory. In the classic text on conformal field theory by Francesco, Mathieu, and S\'en\'echal \cite{Fran1997}, the free-boson partition function (without the zero-mode; see Section 10.2 in \cite{Fran1997}) is given by
 
\begin{equation}\aligned\label{etap}
\mathbf{Z}_{\hbox{bos}}(\tau)=\frac{1}{\sqrt{\Im(\tau)}|\eta(\tau)|^2}.
\endaligned\end{equation}

When the free bosonic theory is compactified in a circle with radius $R$, the corresponding partition function on the torus is given by
\begin{equation}\aligned\label{Zvafa}
\mathbf{Z}(R,\tau)=\frac{R}{\sqrt2}\mathbf{Z}_{\hbox{bos}}(\tau)\sum_{m,m'}\exp(-\frac{\pi R^2|m\tau-m'|^2}{2\Im(\tau)}).
\endaligned\end{equation}
See Section 10.4 in Francesco-Mathieu-S\'en\'echal \cite{Fran1997},
 Bershadsky-Klebanov \cite{Bers1991} and Alvarez-Gaum-Moore-Vafa \cite{Vafa1986} (page 28). By \eqref{thetas}, the widely used partition function \eqref{Zvafa} can be rewritten as

\begin{equation}\aligned\label{Th0}
\mathbf{Z}(R,\tau)=\frac{R}{\sqrt2}\frac{\theta(\frac{R^2}{2},\tau)}{\sqrt{\Im(\tau)}|\eta(\tau)|^2},
\endaligned\end{equation}
which represents a ratio form of theta and Dedekind eta functions (up to a power and scaling).
See \eqref{etap} and \eqref{Th0} also in the classical book of statistical field theory (Mussardo \cite{Muss2020}, chapter 12, pages 404-408)
for the free energy.

The average of partition functions of c massless free bosons in two dimensions over Narain moduli space and
a $U(1)^c\times U(1)^c$ Chern-Simons gauge in three dimensions coupled to topological gravity.
The following three quantities are the same (as summarized by Benjamin-Keller-Ooguri-Zadeh \cite{Ben2022}):

{\it
1. The average partition function of c free bosons $($Afhkami-Jeddi et al. \cite{Afhk2021}$)$ or the average of the genus 1 partition function over the Narain moduli space $($Maloney-Witten \cite{Witten2020}$)$.

\begin{equation}\aligned\label{M1}
\mathbf{Z}_{\mathcal{M}}(\tau)=\frac{\int_{\mathcal{M}}d\mu Z(\mu)}{\int_{\mathcal{M}}d\mu}=\frac{\zeta(\frac{c}{2},\tau)}{(\sqrt{\Im(\tau)}|\eta(\tau)|^2)^c}
,\;\mathcal{M}=O(c,c;\mathbb{Z})\setminus O(c,c)/ O(c)\times O(c).
\endaligned\end{equation}

2. The Poincar\'e sum of a $U(1)^c$ vacuum character.

\begin{equation}\aligned\label{M2}
\mathbf{Z}_{T^c}(\tau)=\sum_{\gamma\in\Gamma_{\infty}\setminus SL(2,\mathbb{Z})}\left| \chi^{vac}(\gamma\tau)\right|^2
=\frac{\zeta(\frac{c}{2},\tau)}{(\sqrt{\Im(\tau)}|\eta(\tau)|^2)^c},\;\; \chi^{vac}(\tau)=\frac{1}{{\eta(\tau)}^c}.
\endaligned\end{equation}

3. An exotic 3d gravity computation of a sum over geometries of a $U(1)^c\times U(1)^c$ abelian Chern-Simons theory:

\begin{equation}\aligned\label{M3}
\mathbf{Z}_{T^c}(\tau)=\sum_{\hbox{3-manifold geometries}} e^{-S_{CS}}
=\frac{\zeta(\frac{c}{2},\tau)}{(\sqrt{\Im(\tau)}|\eta(\tau)|^2)^c}.
\endaligned\end{equation}
}

In the purely mathematical side,
the explicit expression $\mathbf{Z}_{T^c}(\tau)$ in \eqref{M1}-\eqref{M3} is a reformulation of an argument originally by Siegel, and is
known as the Siegel-Weil formula (Benjamin-Keller-Ooguri-Zadeh \cite{Ben2022}).

In the physical and applied side, the explicit expressions $\mathbf{Z}_{T^c}(\tau)$ and $\mathbf{Z}(R,\tau)$ (given by \eqref{Zvafa} and \eqref{Th0}) appeared as partition functions in physical systems. Partition function plays fundamental role in statistical physical: the total energy, free energy, entropy, and pressure, can all be expressed in terms of the partition function or its derivatives. In particular, the Helmholtz free energy ($F$) and the partition function ($\mathbf{Z}$) have the following relation
$$
F=-k_{B}T\log(\mathbf{Z}).
$$
Here $k_B$ is Boltzmann's constant and $T$ is the temperature. Therefore, at a given temperature, locating the min (max) of the partition functions is equivalent to finding the max (min) of the Helmholtz free energy. Since the partition function determines many basic
physical quantities, we are led to the following problem:
\begin{problemA}[Torus geometry and max (min) of partition functions]  How does the geometry of the torus affect the value of partition functions in various physical models?
In particular, what kind of geometry of torus such that the partition functions achieve the extreme values?
\end{problemA}

By Theorems A-C, we have

\begin{thmD}\label{PropZ1}
Assume that $\alpha>0$, $s>1$ and $c>0$. Then, up to the action by the modular group,
$$\argmin_{\tau\in\mathbb{H}}\frac{\theta(\alpha,\tau)}{(\sqrt{\Im(\tau)}|\eta(\tau)|^2)^c}=e^{i\frac{\pi}{3}},\;\;
\argmin_{\tau\in\mathbb{H}}\frac{\zeta(s,\tau)}{(\sqrt{\Im(\tau)}|\eta(\tau)|^2)^c}=e^{i\frac{\pi}{3}}.$$
Namely, the hexagonal lattice minimizes the partition functions in \eqref{Zvafa}, and \eqref{M1}-\eqref{M3}.
\end{thmD}

Motivated by Theorem D and Problem A, it is natural to consider the ratio forms of theta functions and Epstein zeta functions. There are four cases in general by simple combinatorics. We formulate them in the following problem.
\begin{problemB} Assume that $\alpha,\beta>0$ and $s>1$. Classify
 \begin{equation}\aligned
(A): \min_{z\in\mathbb{H}}\frac{\theta(\beta,z)}{\theta(\alpha,z)},
\;\;\text{and}\;\; (B): \min_{z\in\mathbb{H}}\frac{\zeta(s,z)}{\theta(\alpha,z)},\\
(C): \max_{z\in\mathbb{H}}\frac{\theta(\beta,z)}{\theta(\alpha,z)},\;\;\text{and}\;\;
(D): \max_{z\in\mathbb{H}}\frac{\zeta(s,z)}{\theta(\alpha,z)}.
\endaligned\end{equation}
\end{problemB}
Note that the consideration of the ratio between theta functions and certain theta-related functions can be traced back to Lefschetz \cite{L}. Regarding lattice energy, B\'etermin (\cite{Betermin2021SIAM}, p. 1940) considers the maximum of the ratio of a theta function and a theta-related function, where he comments that {\it Optimizing such lattice energy is a huge challenge since the maximizer varies a lot with $x$ and $Y$.} In our situation, the absence of $Y$ (denoting position) makes the problem tractable.

It turns out that the ratio forms in {\bf Problem B} appear frequently in conformal field theory (Francesco-Mathieu-S\'en\'echal \cite{Fran1997}, e.g., Section 10.2). Problem B also arises from partition functions. Noting that ${\mathbf{Z}(R,\tau)}$ denotes the partition function of a free bosonic theory compactified on a circle with radius $R$, and considering the effect of the circle radius on the partition functions, we have
\begin{problemC} Assume that $R, R_1, R_2>0$ and $c>2$. Classify
\begin{equation}\aligned\nonumber
\min_{\tau\in\mathbb{H}}\frac{\mathbf{Z}(R_2,\tau)}{\mathbf{Z}(R_1,\tau)},\;\;\text{and}\;\;
\max_{\tau\in\mathbb{H}}\frac{\mathbf{Z}(R_2,\tau)}{\mathbf{Z}(R_1,\tau)},\;\;\text{and}\;\;
\min_{\tau\in\mathbb{H}}\frac{\mathbf{Z}_{T^c}(\tau)}{\mathbf{Z}^c(R,\tau)}.
\endaligned\end{equation}
\end{problemC}

By \eqref{M2}-\eqref{M3} and \eqref{Th0}, one has
\begin{equation}\aligned\nonumber
\frac{\mathbf{Z}(R_2,\tau)}{\mathbf{Z}(R_1,\tau)}=\frac{\theta(\frac{R_2^2}{2},\tau)}{\theta(\frac{R_1^2}{2},\tau)},\;\;
\frac{\mathbf{Z}_{T^c}(\tau)}{\mathbf{Z}^c(R,\tau)}=(\frac{\sqrt2}{R})^c\frac{\zeta(\frac{c}{2},\tau)}{\theta^c(\frac{R^2}{2},\tau)}.
\endaligned\end{equation}
Problem C coincides partially with Problem B, and is completely solved in Theorems \ref{Th1} and \ref{Th2}.

Through a simple asymptotic analysis, (noting that $\theta (\alpha, iy) \sim \sqrt{y}, \;\zeta (s, iy) \sim y^{s}$ as $y\to +\infty$), it is easy to see that case $(D)$ in Problem B has no maximizer.

To state our results concisely, we denote that
 \begin{equation}\aligned\nonumber
\uppercase\expandafter{\romannumeral1}:=\{(\alpha,\beta):\beta>\alpha, \beta\alpha>1\},\;\;
\uppercase\expandafter{\romannumeral2}:=\{(\alpha,\beta):\beta<\alpha, \beta\alpha>1\},\\
\uppercase\expandafter{\romannumeral3}:=\{(\alpha,\beta):\beta<\alpha, \beta\alpha<1\},\;\;
\uppercase\expandafter{\romannumeral4}:=\{(\alpha,\beta):\beta>\alpha, \beta\alpha<1\}.
\endaligned\end{equation}
A geometric illustration of these regions in the first quadrat can be found in Figure \ref{LJFFF1}.
Our first main concerning cases (A) and (C) in {\bf Problem B}, and we provide a complete characterization as follows:
\begin{figure}
\centering
\includegraphics[scale=0.585]{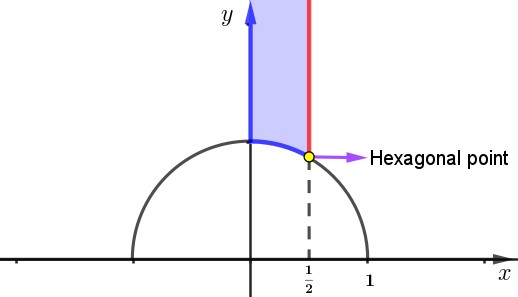}
 \caption{Fundamental domain and hexagonal point}
\label{LJFFF}
\end{figure}
\begin{theorem}[Ratio of theta functions]\label{Th1} Assume that $\alpha>0$ and $\beta>0$. Then, up to the action by the modular group, the following hold.
\begin{itemize}
  \item [(a)] Maximum of ratio of theta functions.
\begin{equation}\aligned\nonumber
\argmax_{z\in\mathbb{H}}\frac{\theta(\beta,z)}{\theta(\alpha,z)}=\begin{cases}
e^{i\frac{\pi}{3}},\;\;\;\;\;\;&\hbox{if}\;\;(\alpha,\beta)\in\uppercase\expandafter{\romannumeral1}\cup\uppercase\expandafter{\romannumeral3};\\
\text{does not exist},\;\;\;\;\;\;&\hbox{if}\;\;(\alpha,\beta)\in\uppercase\expandafter{\romannumeral2}\cup\uppercase\expandafter{\romannumeral4}.
\end{cases}
\endaligned\end{equation}

  \item [(b)]
Minimum of ratio of theta functions.
\begin{equation}\aligned\nonumber
\argmin_{z\in\mathbb{H}}\frac{\theta(\beta,z)}{\theta(\alpha,z)}=\begin{cases}
e^{i\frac{\pi}{3}},\;\;\;\;\;\;&\hbox{if}\;\;(\alpha,\beta)\in\uppercase\expandafter{\romannumeral2}\cup\uppercase\expandafter{\romannumeral4};\\
\text{does not exist},\;\;\;\;\;\;&\hbox{if}\;\;(\alpha,\beta)\in\uppercase\expandafter{\romannumeral1}\cup\uppercase\expandafter{\romannumeral3}.
\end{cases}
\endaligned\end{equation}

\end{itemize}

\end{theorem}

Theorem \ref{Th1} reveals that the hexagonal lattice plays a key role in the ratio of theta functions. We state these results in pure mathematical forms to facilitate generalization and various physical applications.
By Theorem \ref{Th1}, we have the following:
\begin{corollary}[Ratio of theta functions with different powers]\label{Coro1} Assume that $\beta>\alpha\geq1$. Then, up to the action by the modular group, it holds that

\begin{equation}\aligned\nonumber
\argmax_{z\in\mathbb{H}}\frac{\theta(\beta,z)}{\theta^k(\alpha,z)}=\begin{cases}
e^{i\frac{\pi}{3}},\;\;\;\;\;\;&\hbox{if}\;\;k\in[1,\infty);\\
\text{does not exist},\;\;\;\;\;\;&\hbox{if}\;\;k\in(0,1).
\end{cases}
\endaligned\end{equation}

\end{corollary}

\begin{figure}
\centering
\includegraphics[scale=0.685]{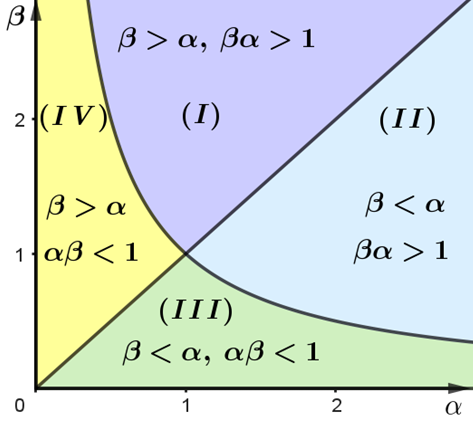}
 \caption{$(\alpha,\beta)$ plane for extreme of $\frac{\theta(\beta,z)}{\theta(\alpha,z)}$}
\label{LJFFF1}
\end{figure}

Theorem \ref{Th1} can be further generalized as follows:
\begin{theorem}[Ratio of sum of theta functions]\label{Coro2} Let $\min_{1\leq j\leq N}\beta_j\geq\max_{1\leq j\leq N}\alpha_j\geq1$ and any $a_j, b_j\geq0$, where $i,j=1\cdots N$ and $N\geq2$ is arbitrary. Then, up to the action by the modular group,
 \begin{equation}\aligned\nonumber
\argmax_{z\in\mathbb{H}}\frac{\sum_{j=1}^Nb_j\theta(\beta_j,z)}{\sum_{j=1}^Na_j\theta(\alpha_j,z)}=e^{i\frac{\pi}{3}}.\;\;
\endaligned\end{equation}
\end{theorem}

Our second main concerning case (B) in {\bf Problem B} with a power $k$. It is stated as follows:

\begin{theorem}[Ratio of Epstein zeta and theta functions]\label{Th2} Assume that $s>1$ and $\alpha\geq3s$. Then, up to the action by the modular group, it holds that:

\begin{equation}\aligned\nonumber
\argmin_{z\in\mathbb{H}}\frac{\zeta(s,z)}{\theta^k(\alpha,z)}=\begin{cases}
e^{i\frac{\pi}{3}},\;\;\;\;\;\;&\hbox{if}\;\;k\in(0, 2s];\\
\text{does not exist},\;\;\;\;\;\;&\hbox{if}\;\;k\in(2s,\infty).
\end{cases}
\endaligned\end{equation}

\end{theorem}

For direct applications in crystallization in lattices (B\'etermin \cite{Bet2016,Betermin2021a}), we state an corollary of Theorem \ref{Th2} in the following.
\begin{corollary}[Differences of Epstein zeta and theta functions with different powers]\label{Coro1.3}

Assume that $ s\in(1,12]$ and $\alpha\geq3s$. Then, up to the action by the modular group,
\begin{equation}\aligned\nonumber
\argmin_{z\in\mathbb{H}}\Big(\zeta(s,z)-\theta^k(\alpha,z)\Big)=\begin{cases}
e^{i\frac{\pi}{3}},\;\;\;\;\;\;&\hbox{if}\;\;k\in(0, 2s];\\
\text{does not exist},\;\;\;\;\;\;&\hbox{if}\;\;k\in(2s,\infty).
\end{cases}
\endaligned\end{equation}

\end{corollary}

\vskip0.1in

The paper is organized as follows: In Section 2, we provide the proof of Theorem \ref{Th1}.
In Section 3, we state a minimum principle for modular invariant functions and collect some summation formulas for Zeta functions.
Finally, we give the proof of Theorem \ref{Th2} and Corollary \ref{Coro1.3} in Section 4.

\section{Proof of Theorem \ref{Th1} and its corollaries}
\setcounter{equation}{0}
Define the subgroup of the modular group as follows:
\begin{equation}\aligned\label{GroupG1}
\mathcal{G}: \hbox{the group generated by} \;\;\tau\mapsto -\frac{1}{\tau},\;\; \tau\mapsto \tau+1,\;\;\tau\mapsto -\overline{\tau}.
\endaligned\end{equation}
The fundamental domain associated to the group $\mathcal{G}$ is given by
\begin{equation}\aligned\nonumber
\mathcal{D}_{\mathcal{G}}:=\{
z\in\mathbb{H}: |z|>1,\; 0<\Re(z)<\frac{1}{2}
\}.
\endaligned\end{equation}
It is well known that theta and Epstein zeta functions are $\mathcal{G}$-invariant.
For convenience, we define
\begin{equation}\aligned\label{Gfff}
&\Gamma_1:=\{
z\in\mathbb{H}: \Re(z)=0,\; \Im(z)\geq1
\},\\
&\Gamma_2:=\{
z\in\mathbb{H}: z=e^{i\theta},\; \theta\in[\pi/3,\pi/2]
\},\\
&\Gamma_3:=\{
z\in\mathbb{H}: \Re(z)=1/2,\; \Im(z)\geq{\sqrt3}/{2}
\}.
\endaligned\end{equation}

First we show that, by deformation, Theorem \ref{Th1} follows from the following
\begin{theorem}\label{Th12} Assume that $\beta>\alpha\geq1$. Then
\begin{itemize}
  \item  [(1)] $
\argmax_{z\in\mathbb{H}}\frac{\theta(\beta,z)}{\theta(\alpha,z)}=e^{i\frac{\pi}{3}}$.

  \item [(2)]  $\argmin_{z\in\mathbb{H}}\frac{\theta(\beta,z)}{\theta(\alpha,z)}$ does not exist.
 \end{itemize}
\end{theorem}

In fact,  in the case (a) of Theorem \ref{Th1}: $\beta>\alpha, \beta\alpha>1$, we consider  two subcases, (a1): $\beta>\alpha\geq1$ and (a2): $\beta>\frac{1}{\alpha}\geq1$.
The subcases (a1) are exactly proved in Theorem \ref{Th12}. For subcases (a2), we use the deformation
$\frac{\theta(\beta,z)}{\theta(\alpha,z)}=\alpha\frac{\theta(\beta,z)}{\theta(\frac{1}{\alpha},z)}$, then it reduces to Theorem \ref{Th12}.
The case (b): $\beta<\alpha, \beta\alpha<1$ in Theorem \ref{Th1} contains two subcases, namely, (b1): $\frac{1}{\beta}>\frac{1}{\alpha}\geq1$ and (b2): $\frac{1}{\beta}>\alpha\geq1$. In subcases (b1) and (b2), one uses the deformations $\frac{\theta(\beta,z)}{\theta(\alpha,z)}=\frac{\alpha}{\beta}\frac{\theta(\frac{1}{\beta},z)}{\theta(\frac{1}{\alpha},z)}$ and
$\frac{\theta(\beta,z)}{\theta(\alpha,z)}=\frac{1}{\beta}\frac{\theta(\frac{1}{\beta},z)}{\theta(\alpha,z)}$ respectively, then they are reduced to Theorem \ref{Th12}. The case (c): $\beta<\alpha, \beta\alpha>1$ in Theorem \ref{Th1} contains two subcases, (c1): $\alpha>\beta\geq1$
and (c2): $\alpha>\frac{1}{\beta}\geq1$. In subcases (c1) and (c2), one uses the deformations $\frac{\theta(\beta,z)}{\theta(\alpha,z)}=\frac{1}{\frac{\theta(\alpha,z)}{\theta(\beta,z)}}$ and
$\frac{\theta(\beta,z)}{\theta(\alpha,z)}=\frac{1}{\beta}\frac{1}{\frac{\theta(\alpha,z)}{\theta(\frac{1}{\beta},z)}}$ respectively, then  they are reduced to Theorem \ref{Th12}. Similar analysis applied to case (d): $\beta>\alpha, \beta\alpha<1$ in Theorem \ref{Th1}, we omit the details here.

We now prove Theorem \ref{Th12}. We first state a preliminary extreme property, which can be viewed as a consequence of Proposition \ref{Prop1A}.

\begin{proposition}\label{PropA}
Assume that $\beta>\alpha\geq1$. Then

 \begin{equation}\aligned\nonumber
\argmax_{z\in\Gamma_2}\frac{\theta(\beta,z)}{\theta(\alpha,z)}=e^{i\frac{\pi}{3}},\;
\argmin_{z\in\Gamma_2}\frac{\theta(\beta,z)}{\theta(\alpha,z)}=i.
\endaligned\end{equation}
\end{proposition}

Proposition \ref{PropA} is proved by the following lemma:
\begin{lemma}\label{LemmaA}
Assume that $\beta>\alpha\geq1$. Then

 \begin{equation}\aligned\nonumber
\partial_\theta\frac{\theta(\beta,e^{i\theta})}{\theta(\alpha,e^{i\theta})}\leq0,\;\;\hbox{for}\;\;\theta\in[\pi/3,\pi/2].
\endaligned\end{equation}
\end{lemma}

\begin{proof}
Fix $\beta>\alpha\geq1$. By Proposition \ref{Prop1A},
 \begin{equation}\aligned\label{2Q1}
\partial_x\frac{\theta(\beta,e^{i\theta})}{\theta(\alpha,e^{i\theta})}\geq0,\;\;\hbox{for}\;\;\theta\in[\pi/3,\pi/2].
\endaligned\end{equation}
Since $\frac{\theta(\beta,z)}{\theta(\alpha,z)}$ is $\mathcal{G}$-invariant, then
 $\frac{\theta(\beta,re^{i\theta})}{\theta(\alpha,re^{i\theta})}=\frac{\theta(\beta,1/re^{i\theta})}{\theta(\alpha,1/re^{i\theta})}$ for $r>0$.
After taking derivative with respect to $r$ and evaluating at $r=1$, we get
 \begin{equation}\aligned\label{2Q2}
\partial_x\frac{\theta(\beta,e^{i\theta})}{\theta(\alpha,e^{i\theta})}
\cos\theta+\partial_y\frac{\theta(\beta,e^{i\theta})}{\theta(\alpha,e^{i\theta})}
\sin\theta=0.
\endaligned\end{equation}
Thus, \eqref{2Q1} and \eqref{2Q2} yield that
 \begin{equation}\aligned\label{2Q3}
\partial_y\frac{\theta(\beta,e^{i\theta})}{\theta(\alpha,e^{i\theta})}\leq0,\;\;\hbox{for}\;\;\theta\in[\pi/3,\pi/2].
\endaligned\end{equation}
On the other hand, by taking derivative with respect to $\theta$, we have
 \begin{equation}\aligned\label{2Q4}
\partial_\theta\frac{\theta(\beta,e^{i\theta})}{\theta(\alpha,e^{i\theta})}
=\partial_x\frac{\theta(\beta,e^{i\theta})}{\theta(\alpha,e^{i\theta})}(-\sin\theta)
+\partial_y\frac{\theta(\beta,e^{i\theta})}{\theta(\alpha,e^{i\theta})}\cos\theta.
\endaligned\end{equation}
In view of \eqref{2Q4}, \eqref{2Q1} and \eqref{2Q3} yield the result.

\end{proof}

We outline the proof of Theorem \ref{Th12} into two main steps.

In Step one, we show that the maximizer can be reduced to the vertical line $\Gamma$.
We shall prove that
 \begin{equation}\aligned
\max_{z\in\mathbb{H}}\frac{\theta(\beta,z)}{\theta(\alpha,z)}
=\max_{z\in\overline{\mathcal{D}_{\mathcal{G}}}}\frac{\theta(\beta,z)}{\theta(\alpha,z)}
=\max_{z\in\Gamma_3}\frac{\theta(\beta,z)}{\theta(\alpha,z)}\;\;\hbox{for}\;\;\beta>\alpha\geq1.
\endaligned\end{equation}
This is a consequence of Proposition \ref{Prop1A}. Propositions \ref{Prop1A} and \ref {PropA} imply that
assuming the existence of the minimizers, we have
\begin{equation}\aligned
\min_{z\in\mathbb{H}}\frac{\theta(\beta,z)}{\theta(\alpha,z)}
=\min_{z\in\overline{\mathcal{D}_{\mathcal{G}}}}\frac{\theta(\beta,z)}{\theta(\alpha,z)}
=\min_{z\in\Gamma_1\cup\Gamma_2}\frac{\theta(\beta,z)}{\theta(\alpha,z)}
=\min_{z\in\Gamma_1}\frac{\theta(\beta,z)}{\theta(\alpha,z)}
\;\;\hbox{for}\;\;\beta>\alpha\geq1.
\endaligned\end{equation}

In Step two, we show that the maximizer is located on $z=\frac{1}{2}+i\frac{\sqrt3}{2}=e^{i\frac{\pi}{3}}$.
We shall prove that
 \begin{equation}\aligned
\max_{z\in\Gamma_3}\frac{\theta(\beta,z)}{\theta(\alpha,z)}\;\;\hbox{is achieved at}\;\;\frac{1}{2}+i\frac{\sqrt3}{2}\;\;\hbox{for}\;\;\beta>\alpha\geq1.
\endaligned\end{equation}

This follows from Proposition \ref{Prop1B}. By Proposition \ref{Prop1C} we have
\begin{equation}\aligned
\min_{z\in\Gamma_1}\frac{\theta(\beta,z)}{\theta(\alpha,z)}\;\;\hbox{does not exist for}\;\;\beta>\alpha\geq1.
\endaligned\end{equation}

Combining Steps one and two, we complete the proof of Theorem \ref{Th12}.

In the remaining part, we prove these propositions  used in proof of Theorem \ref{Th12}.

\subsection{Transversal monotonicity}

In this subsection, we aim to prove a transversal monotonicity on ratio of theta functions.
It is stated as follows
\begin{proposition}\label{Prop1A} Assume that $\beta>\alpha\geq1$. Then
\begin{equation}\aligned\nonumber
\frac{\partial}{\partial x}\frac{\theta(\beta,z)}{\theta(\alpha,z)}\geq0\;\;\hbox{for}\;\;z\in\overline{\mathcal{D}_{\mathcal{G}}}.
\endaligned\end{equation}

\end{proposition}

The proof of Proposition \ref{Prop1A} will be given at the end of this subsection. Before that we shall prove some preliminary lemmas first.

In terms of one dimensional theta function, one has an alternative expression of theta functions.
\begin{lemma}\label{Lemma21} Assume that $z\in \mathbb{H}$ and $\alpha>0$. Then
\begin{equation}\aligned\label{CM3}
\sum_{n\in\mathbb{Z}} e^{-\pi\alpha yn^2}\vartheta(\frac{y}{\alpha};nx)=
\sqrt{\frac{\alpha}{y}}\cdot\theta(\alpha,z).
\endaligned\end{equation}
Here the classical one-dimensional theta function  is given by
\begin{equation}\aligned\label{TXY}
\vartheta(X;Y):=\sum_{n\in\mathbb{Z}} e^{-\pi n^2 X} e^{2n\pi i Y},\;\hbox{where}\; X>0,\; Y\in\R.
 \endaligned\end{equation}
\end{lemma}

Recall that
\begin{lemma}[Montgomery's first lemma \cite{Mon1988}]\label{Lemma22} Assume that $\alpha\geq1$. Then

\begin{equation}\aligned\nonumber
\frac{\partial}{\partial x}\theta(\alpha,z)\leq0\;\;\hbox{for}\;\;z\in\mathcal{D}_{\mathcal{G}}.
\endaligned\end{equation}
Or equivalently,
\begin{equation}\aligned\nonumber
\frac{\partial}{\partial x}
\sum_{n\in\mathbb{Z}} e^{-\pi\alpha yn^2}\vartheta(\frac{y}{\alpha};nx)\leq0\;\;\hbox{for}\;\;z\in\mathcal{D}_{\mathcal{G}}.
\endaligned\end{equation}

\end{lemma}

In our previous work \cite{Luo2023a}, we have established that
\begin{lemma}[Corollary of Theorem 3.4 in \cite{Luo2023a}]\label{Lemma23}
Assume that $s\geq1$. Then
\begin{equation}\aligned\nonumber
\frac{\partial}{\partial x}\frac{\partial}{\partial s}(\sqrt{s}\theta(s,z))\geq0\;\;\hbox{for}\;\;z\in\mathcal{D}_{\mathcal{G}}.
\endaligned\end{equation}

\end{lemma}

Using Lemma \ref{Lemma23} and fundamental theorem of calculus, one has
\begin{equation}\aligned\nonumber
\sqrt{\beta}\theta(\beta,z)-\sqrt{\alpha}\theta(\alpha,z)=
\int_\alpha^\beta \frac{\partial}{\partial s}(\sqrt{s}\theta(s,z))ds.
\endaligned\end{equation}
Then
\begin{equation}\aligned\label{F123}
\frac{\partial}{\partial x}\Big(\sqrt{\beta}\theta(\beta,z)-\sqrt{\alpha}\theta(\alpha,z)\Big)=
\int_\alpha^\beta \frac{\partial}{\partial x}\frac{\partial}{\partial s}(\sqrt{s}\theta(s,z))ds.
\endaligned\end{equation}

Therefore, by Lemma \ref{Lemma23} and \eqref{F123}, it holds that

\begin{lemma}\label{Lemma24} Assume that $\beta>\alpha\geq1$. Then
\begin{equation}\aligned\label{nonumber}
\frac{\partial}{\partial x}\Big(
\sqrt{\beta}\theta(\beta,z)-\sqrt{\alpha}\theta(\alpha,z)\Big)\geq0
\;\;\hbox{for}\;\;z\in\mathcal{D}_{\mathcal{G}}.
\endaligned\end{equation}
Or equivalently,
\begin{equation}\aligned\label{CM2}
\frac{\partial}{\partial x}\Big(\sum_{n\in\mathbb{Z}} e^{-\pi\beta yn^2}\vartheta(\frac{y}{\beta};nx)-
\sum_{n\in\mathbb{Z}} e^{-\pi\alpha yn^2}\vartheta(\frac{y}{\alpha};nx)\Big)\geq0\;\;\hbox{for}\;\;z\in\mathcal{D}_{\mathcal{G}}.
\endaligned\end{equation}
\end{lemma}

We shall also prove that

\begin{lemma}\label{Lemma25} Assume that $\beta>\alpha\geq1$. Then

\begin{equation}\aligned\nonumber
\sqrt\beta\theta(\beta,z)\geq\sqrt\alpha\theta(\alpha,z)\;\;\hbox{for}\;\;z\in\mathcal{D}_{\mathcal{G}}.
\endaligned\end{equation}
Or equivalently,
\begin{equation}\aligned\label{CM}
\sum_{n\in\mathbb{Z}} e^{-\pi\beta yn^2}\vartheta(\frac{y}{\beta};nx)\geq
\sum_{n\in\mathbb{Z}} e^{-\pi\alpha yn^2}\vartheta(\frac{y}{\alpha};nx)\;\;\hbox{for}\;\;z\in\mathcal{D}_{\mathcal{G}}.
\endaligned\end{equation}

\end{lemma}

By Lemma \ref{Lemma24}, to prove Lemma \ref{Lemma25}, it suffices to prove that
$\sqrt\beta\theta(\beta,z)\geq\sqrt\alpha\theta(\alpha,z)$ on the left boundary of half fundamental domain $\mathcal{D}_{\mathcal{G}}$.
These are done in Lemmas \ref{Lemma26} and \ref{Lemma27}.

\begin{lemma}\label{Lemma26} Assume that $\beta>\alpha\geq1$. Then
\begin{equation}\aligned\nonumber
\sqrt\beta\theta(\beta,z)\mid_{\Re(z)=0}\geq\sqrt\alpha\theta(\alpha,z)\mid_{\Re(z)=0}\;\;\hbox{for}\;\;\Im(z)\geq1.
\endaligned\end{equation}

\begin{lemma}\label{Lemma27} Assume that $\beta>\alpha\geq1$. Then
\begin{equation}\aligned\nonumber
\sqrt\beta\theta(\beta,z)\mid_{|z|=1, 0\leq\Re(z)\leq\frac{1}{2}}\geq\sqrt\alpha\theta(\alpha,z)\mid_{|z|=1, 0\leq\Re(z)\leq\frac{1}{2}}.
\endaligned\end{equation}
\end{lemma}

\end{lemma}

By Lemma \ref{Lemma21},
\begin{equation}\aligned\nonumber
\sum_{n\in\mathbb{Z}} e^{-\pi\alpha yn^2}\vartheta(\frac{y}{\alpha};nx)\mid_{x=0}
=\vartheta_3(\alpha y)\vartheta_3(\frac{y}{\alpha}).
\endaligned\end{equation}
Then we have
\begin{lemma}[Evaluation of theta function on $y$-axis]\label{Lemma28}

 \begin{equation}\aligned\nonumber
\theta(\alpha,iy)=\begin{cases}\sqrt{\frac{y}{\alpha}}\vartheta_3(\alpha y)\vartheta_3(\frac{y}{\alpha})\;&\hbox{for}\;\;\frac{y}{\alpha}\;\;\hbox{has a positive lower bound},\\
\vartheta_3(\alpha y)\vartheta_3(\frac{\alpha}{y})\;&\hbox{for}\;\;\frac{\alpha}{y}\;\;\hbox{has a positive lower bound}.
\end{cases}
\endaligned\end{equation}
Here $\vartheta_3$ is the Jacobi theta function of third type and defined as
 \begin{equation}\aligned\nonumber
\vartheta_3(x):=\sum_{n\in\mathbb{Z}}e^{-\pi n^2 x}.
\endaligned\end{equation}

\end{lemma}

By Lemma \ref{Lemma28}, Lemma \ref{Lemma26} is equivalent to

\begin{lemma}[=Lemma \ref{Lemma26}]\label{Lemma29} Assume that $\beta>\alpha\geq1$. Then
\begin{equation}\aligned\nonumber
\vartheta_3(\beta y)\vartheta_3(\frac{y}{\beta})
\geq
\vartheta_3(\alpha y)\vartheta_3(\frac{y}{\alpha})\;\;\hbox{for}\;\;y\geq1.
\endaligned\end{equation}

\end{lemma}

To prove Lemma \ref{Lemma29}, it suffices to prove that

\begin{lemma}\label{Lemma210} Assume that $\alpha\geq1$. Then
\begin{equation}\aligned\nonumber
\frac{\partial}{\partial \alpha}
\Big(\vartheta_3(\alpha y)\vartheta_3(\frac{y}{\alpha})\Big)\geq0\;\;\hbox{for}\;\;y\geq1.
\endaligned\end{equation}

\end{lemma}

By symmetry, Lemma \ref{Lemma210} is equivalent to
\begin{lemma}\label{Lemma211} Assume that $\alpha\geq1$. Then
\begin{equation}\aligned\nonumber
\frac{\partial}{\partial y}
\Big(\vartheta_3(\alpha y)\vartheta_3(\frac{\alpha}{y})\Big)\geq0\;\;\hbox{for}\;\;y\geq1.
\endaligned\end{equation}

\end{lemma}
By Lemma \ref{Lemma28}, Lemma \ref{Lemma211} is equivalent to following Montegomery's Lemma \cite{Mon1988}.
\begin{lemma}[Montegomery's second Lemma \cite{Mon1988}]\label{Lemma212} Assume that $\alpha\geq1$. Then
\begin{equation}\aligned\nonumber
\frac{\partial}{\partial y}
\theta(\alpha,z)\geq0\;\;\hbox{for}\;\;z\in\mathcal{D}_{\mathcal{G}}.
\endaligned\end{equation}

\end{lemma}

Therefore, Lemma \ref{Lemma26} is proved. It remains to prove Lemma \ref{Lemma27}.
By the group invariance($z\mapsto\frac{1}{1-z}$), one has

\begin{lemma}[From arc to $\frac{1}{2}-$vertical line] \label{Lemma213} Assume that $\alpha, \beta>0$, it holds that
\begin{equation}\aligned\nonumber
\sqrt\beta\theta(\beta,z)-\sqrt\alpha\theta(\alpha,z)\mid_{|z|=1, \Re(z)\in[0,\frac{1}{2}]}
=\sqrt\beta\theta(\beta,\frac{1}{2}+iy')-\sqrt\alpha\theta(\alpha,\frac{1}{2}+iy'),\;  y'\in[\frac{1}{2},\frac{\sqrt3}{2}],
\endaligned\end{equation}
explicitly, $y'=\frac{1}{2}\sqrt\frac{1+\Re(z)}{1-\Re(z)}$. In particular,
\begin{equation}\aligned\label{P100}
\sqrt\beta\theta(\beta,i)-\sqrt\alpha\theta(\alpha,i)
=\sqrt\beta\theta(\beta,\frac{1}{2}+i\frac{1}{2})-\sqrt\alpha\theta(\alpha,\frac{1}{2}+i\frac{1}{2}).
\endaligned\end{equation}

\end{lemma}

In fact, one has

\begin{lemma}\label{Lemma214}
 Assume that $\alpha, \beta>0$, it holds that
\begin{equation}\aligned\nonumber
\frac{\partial}{\partial y}\Big(\sqrt\beta\theta(\beta,\frac{1}{2}+iy)-\sqrt\alpha\theta(\alpha,\frac{1}{2}+iy)
\Big)\geq0\;\;\hbox{for}\;\;
y\in[\frac{1}{2},\frac{\sqrt3}{2}].
\endaligned\end{equation}
\end{lemma}

Lemma \ref{Lemma214} is proved by fundamental theorem of calculus
\begin{equation}\aligned\nonumber
\sqrt\beta\theta(\beta,\frac{1}{2}+iy)-\sqrt\alpha\theta(\alpha,\frac{1}{2}+iy)
=\int_\alpha^\beta\frac{\partial}{\partial s}\Big(\sqrt s\theta(s,\frac{1}{2}+iy)\Big)ds.
\endaligned\end{equation}

\begin{lemma}[(2) of Lemma \ref{LemmaPPP}]\label{Lemma215}
For $s\geq1$,
\begin{equation}\aligned\nonumber
\frac{\partial}{\partial y}\frac{\partial}{\partial s}\Big(\sqrt s\theta(s,\frac{1}{2}+iy)\Big)\geq0\;\;\hbox{for}\;\;y\in[\frac{1}{2},\frac{\sqrt3}{2}].
\endaligned\end{equation}

\end{lemma}

On the other hand, by Lemma \ref{Lemma26},
\begin{equation}\aligned\label{P101}
\sqrt\beta\theta(\beta,i)-\sqrt\alpha\theta(\alpha,i)
\geq0\;\;\hbox{for}\;\;\beta>\alpha\geq1.
\endaligned\end{equation}
This and \eqref{P100} in Lemma \ref{Lemma213} implies that
\begin{equation}\aligned\label{P102}
\sqrt\beta\theta(\beta,\frac{1}{2}+i\frac{1}{2})-\sqrt\alpha\theta(\alpha,\frac{1}{2}+i\frac{1}{2})\geq0\;\;\hbox{for}\;\;\beta>\alpha\geq1.
\endaligned\end{equation}
Therefore, \eqref{P102} and Lemmas \ref{Lemma213}, \ref{Lemma214} yield Lemma \ref{Lemma27}.

 We are in a position to prove the main result (Proposition \ref{Prop1A}) in this subsection.
\begin{proof}{\bf Proof of Proposition \ref{Prop1A}.}
The key is to use a new but equivalent ratio form,
$$\frac{\sqrt\beta\theta(\beta,z)}{\sqrt\alpha\theta(\alpha,z)}.$$
It suffices to prove that
\begin{equation}\aligned\nonumber
\frac{\partial}{\partial x}\frac{\sqrt\beta\theta(\beta,z)}{\sqrt\alpha\theta(\alpha,z)}\geq0\;\;\hbox{for}\;\;z\in\mathcal{D}_{\mathcal{G}}.
\endaligned\end{equation}
A direct calculation shows that

\begin{equation}\aligned\nonumber
\frac{\partial}{\partial x}\frac{\sqrt\beta\theta(\beta,z)}{\sqrt\alpha\theta(\alpha,z)}=
\frac{\frac{\partial}{\partial x}(\sqrt\beta\theta(\beta,z))\sqrt\alpha\theta(\alpha,z)
-\frac{\partial}{\partial x}(\sqrt\alpha\theta(\alpha,z))\sqrt\beta\theta(\beta,z)
}
{\alpha \theta^2(\alpha,z)}.
\endaligned\end{equation}
Then it is also equivalent to proving that

\begin{equation}\aligned\label{K0}
\frac{\partial}{\partial x}(\sqrt\beta\theta(\beta,z))\sqrt\alpha\theta(\alpha,z)
-\frac{\partial}{\partial x}(\sqrt\alpha\theta(\alpha,z))\sqrt\beta\theta(\beta,z)\geq0\;\;\hbox{for}\;\;z\in\mathcal{D}_{\mathcal{G}}.
\endaligned\end{equation}

Regrouping the terms, we get that
\begin{equation}\aligned\nonumber
&\frac{\partial}{\partial x}(\sqrt\beta\theta(\beta,z))\sqrt\alpha\theta(\alpha,z)
-\frac{\partial}{\partial x}(\sqrt\alpha\theta(\alpha,z))\sqrt\beta\theta(\beta,z)\\
=&\frac{\partial}{\partial x}(\sqrt\beta\theta(\beta,z))\sqrt\alpha\theta(\alpha,z)
-\frac{\partial}{\partial x}(\sqrt\beta\theta(\beta,z))\sqrt\beta\theta(\beta,z)\\
&+\frac{\partial}{\partial x}(\sqrt\beta\theta(\beta,z))\sqrt\beta\theta(\beta,z)
-\frac{\partial}{\partial x}(\sqrt\alpha\theta(\alpha,z))\sqrt\beta\theta(\beta,z).
\endaligned\end{equation}

Then it  holds

\begin{equation}\aligned\nonumber
&\frac{\partial}{\partial x}(\sqrt\beta\theta(\beta,z))\sqrt\alpha\theta(\alpha,z)
-\frac{\partial}{\partial x}(\sqrt\alpha\theta(\alpha,z))\sqrt\beta\theta(\beta,z)\\
=&\sqrt\beta\frac{\partial}{\partial x}(\theta(\beta,z))
\Big(
\sqrt\alpha\theta(\alpha,z)
-\sqrt\beta\theta(\beta,z)
\Big)
+\sqrt\beta\theta(\beta,z)\frac{\partial}{\partial x}\Big(
\sqrt\beta\theta(\beta,z)-\sqrt\alpha\theta(\alpha,z)
\Big).
\endaligned\end{equation}
To simplify the expression, let
\begin{equation}\aligned\nonumber
\mathcal{B}_a(\alpha,\beta,z):&=\sqrt\beta\frac{\partial}{\partial x}(\theta(\beta,z))
\Big(
\sqrt\alpha\theta(\alpha,z)
-\sqrt\beta\theta(\beta,z)
\Big)\\
\mathcal{B}_b(\alpha,\beta,z):&=
\sqrt\beta\theta(\beta,z)\frac{\partial}{\partial x}\Big(
\sqrt\beta\theta(\beta,z)-\sqrt\alpha\theta(\alpha,z)
\Big).
\endaligned\end{equation}

Then

\begin{equation}\aligned\label{K1}
\frac{\partial}{\partial x}(\sqrt\beta\theta(\beta,z))\sqrt\alpha\theta(\alpha,z)
-\frac{\partial}{\partial x}(\sqrt\alpha\theta(\alpha,z))\sqrt\beta\theta(\beta,z)
=\mathcal{B}_a(\alpha,\beta,z)+\mathcal{B}_b(\alpha,\beta,z).
\endaligned\end{equation}

On the other hand, by Lemmas \ref{Lemma22} and \ref{Lemma25},
\begin{equation}\aligned\label{K2}
\mathcal{B}_a(\alpha,\beta,z)\geq0.
\endaligned\end{equation}

And similarly by Lemma \ref{Lemma24},

\begin{equation}\aligned\label{K3}
\mathcal{B}_b(\alpha,\beta,z)\geq0.
\endaligned\end{equation}
\eqref{K1}, \eqref{K2} and \eqref{K3} yield \eqref{K0}. These complete the proof.

\end{proof}

\subsection{Monotonicity on the $\frac{1}{2}-$Vertical line}

In this subsection, we aim to prove that
\begin{proposition}\label{Prop1B} Assume that $\beta>\alpha\geq1$. Then
\begin{equation}\aligned\nonumber
\frac{\partial}{\partial y}\frac{\theta(\beta,z)}{\theta(\alpha,z)}\mid_{\Re(z)=\frac{1}{2}}\leq0\;\;\hbox{for}\;\;\Im(z)\geq\frac{\sqrt3}{2}.
\endaligned\end{equation}

\end{proposition}
To prove Proposition \ref{Prop1B}, we establish one more auxiliary lemma except those in previous subsection.
\begin{lemma}\label{Lemma216} Assume that $\beta>\alpha\geq1$. Then
\begin{equation}\aligned\nonumber
\frac{\partial}{\partial y}\Big(\sqrt{\beta}\theta(\beta,z)-\sqrt{\alpha}\theta(\alpha,z)\Big)\mid_{\Re(z)=\frac{1}{2}}\leq0\;\;\hbox{for}\;\;\Im(z)\geq\frac{\sqrt3}{2}.
\endaligned\end{equation}

\end{lemma}
Via the deformation,
\begin{equation}\aligned\nonumber
\frac{\partial}{\partial y}\Big(\sqrt{\beta}\theta(\beta,z)-\sqrt{\alpha}\theta(\alpha,z)\Big)
=\frac{\partial}{\partial y}\int_\alpha^\beta \frac{\partial}{\partial s}(\sqrt{s}\theta(s,z))ds\\
=\int_\alpha^\beta \frac{\partial^2}{\partial y\partial s}(\sqrt{s}\theta(s,z))ds,
\endaligned\end{equation}
Lemma \ref{Lemma216} is deduced by item (1) in Lemma \ref{LemmaPPP}, which is proved by our previous paper \cite{Luo2023a}. In fact,
item (1) in Lemma \ref{LemmaPPP} is followed by Proposition 4.1 in \cite{Luo2023a}, the proof of items (2), (3) is similar, hence we omit the detail here.
\begin{lemma}[\cite{Luo2023a}]\label{LemmaPPP} Assume that $s\geq1$. Then
\begin{itemize}
  \item [(1)] $
\frac{\partial^2}{\partial y\partial s}\Big(\sqrt{s}\theta(s,z)\Big)\mid_{\Re(z)=\frac{1}{2}}\leq0\;\;\hbox{for}\;\;\Im(z)\geq\frac{\sqrt3}{2}.
$
  \item [(2)] $
\frac{\partial^2}{\partial y\partial s}\Big(\sqrt{s}\theta(s,z)\Big)\mid_{\Re(z)=\frac{1}{2}}\geq0\;\;\hbox{for}\;\;\Im(z)\in[\frac{1}{2},\frac{\sqrt3}{2}].
$

  \item [(3)] $
\frac{\partial^2}{\partial y\partial s}\Big(\sqrt{s}\theta(s,z)\Big)\mid_{\Re(z)=0}\leq0\;\;\hbox{for}\;\;\Im(z)\geq1.
$
\end{itemize}

\end{lemma}

\begin{proof}{\bf Proof of Proposition \ref{Prop1B}.} Using the deformation,
\begin{equation}\aligned\nonumber
\sqrt{\frac{\beta}{\alpha}}\cdot\frac{\partial}{\partial y}\frac{\theta(\beta,z)}{\theta(\alpha,z)}
=\frac{\partial}{\partial y}\frac{\sqrt{\beta}\theta(\beta,z)}{\sqrt{\alpha}\theta(\alpha,z)}.
\endaligned\end{equation}
A direct calculation and deformation show that
\begin{equation}\aligned\nonumber
&(\sqrt{\alpha}\theta(\alpha,z))^2\cdot\frac{\partial}{\partial y}\frac{\sqrt{\beta}\theta(\beta,z)}{\sqrt{\alpha}\theta(\alpha,z)}\\
=&\frac{\partial}{\partial y}(\sqrt{\beta}\theta(\beta,z))\cdot(\sqrt{\alpha}\theta(\alpha,z))
-\frac{\partial}{\partial y}(\sqrt{\alpha}\theta(\alpha,z))\cdot(\sqrt{\beta}\theta(\beta,z))\\
=&\Big(\frac{\partial}{\partial y}(\sqrt{\beta}\theta(\beta,z))\cdot(\sqrt{\alpha}\theta(\alpha,z))
-\frac{\partial}{\partial y}(\sqrt{\alpha}\theta(\alpha,z))\cdot(\sqrt{\alpha}\theta(\alpha,z))\Big)\\
+&\Big(\frac{\partial}{\partial y}(\sqrt{\alpha}\theta(\alpha,z))\cdot(\sqrt{\alpha}\theta(\alpha,z))
-\frac{\partial}{\partial y}(\sqrt{\alpha}\theta(\alpha,z))\cdot(\sqrt{\beta}\theta(\beta,z))\Big)\\
=&\sqrt{\alpha}\theta(\alpha,z)\cdot\frac{\partial}{\partial y}\Big(\sqrt{\beta}\theta(\beta,z)-\sqrt{\alpha}\theta(\alpha,z)\Big)
+\frac{\partial}{\partial y}(\sqrt{\alpha}\theta(\alpha,z))\cdot\Big(\sqrt{\alpha}\theta(\alpha,z)-\sqrt{\beta}\theta(\beta,z)\Big).
\endaligned\end{equation}
For convenience, we denote that
\begin{equation}\aligned\nonumber
\mathcal{H}_a(\alpha,\beta,z):&=\sqrt{\alpha}\theta(\alpha,z)\cdot\frac{\partial}{\partial y}\Big(\sqrt{\beta}\theta(\beta,z)-\sqrt{\alpha}\theta(\alpha,z)\Big),\\
\mathcal{H}_b(\alpha,\beta,z):&=\frac{\partial}{\partial y}(\sqrt{\alpha}\theta(\alpha,z))\cdot\Big(\sqrt{\alpha}\theta(\alpha,z)-\sqrt{\beta}\theta(\beta,z)\Big).
\endaligned\end{equation}
Then
\begin{equation}\aligned\nonumber
(\sqrt{\alpha}\theta(\alpha,z))^2\cdot\frac{\partial}{\partial y}\frac{\sqrt{\beta}\theta(\beta,z)}{\sqrt{\alpha}\theta(\alpha,z)}
=\mathcal{H}_a(\alpha,\beta,z)+\mathcal{H}_b(\alpha,\beta,z).
\endaligned\end{equation}
By item (1) of Lemma \ref{LemmaPPP}, if $\beta>\alpha\geq1$, then
\begin{equation}\aligned\nonumber
\mathcal{H}_a(\alpha,\beta,z)\mid_{\Re(z)=\frac{1}{2}}\leq0\;\;\hbox{for}\;\;\Im(z)\geq\frac{\sqrt3}{2}.
\endaligned\end{equation}

By Lemmas \ref{Lemma25} and \ref{Lemma212}, if $\beta>\alpha\geq1$, then
\begin{equation}\aligned\nonumber
\mathcal{H}_b(\alpha,\beta,z)\mid_{\Re(z)=\frac{1}{2}}\leq0\;\;\hbox{for}\;\;\Im(z)\geq\frac{\sqrt3}{2}.
\endaligned\end{equation}
\end{proof}

Similar to the proof of Proposition \ref{Prop1B}, using Lemmas \ref{Lemma25}, \ref{Lemma212} and \ref{LemmaPPP} (item (3)), we have

\begin{proposition}\label{Prop1C} Assume that $\beta>\alpha\geq1$. Then
\begin{equation}\aligned\nonumber
\frac{\partial}{\partial y}\frac{\theta(\beta,z)}{\theta(\alpha,z)}\mid_{z=iy, \;y\geq1}\leq0.
\endaligned\end{equation}

\end{proposition}

%%%%%%%%%%%%%%%%%%%%%%%%%%%%%%%%%%%%%%%%%%%%%%%%%%%%%%%%%%%%%%%%%%%%%%%%%%%%%%%%%%%%%%%%%%%%%%%%%%%%
%%%%%%%%%%%%%%%%%%%%%%%%%%%%%%%%%%%%%%%%%%%%%%%%%%%%%%%%%%%%%%%%%%%%%%%%%%%%%%%%%%%%%%%%%%%%%%%%%%%%

\subsection{Proof of Corollary \ref{Coro1} and Theorem \ref{Coro2}}

Proof of Corollary \ref{Coro1}.
For $k\geq1$, we use the deformation as follows
\begin{equation}\aligned\nonumber
\frac{\theta(\beta,z)}{\theta^k(\alpha,z)}=\frac{\theta(\beta,z)}{\theta(\alpha,z)}\cdot
\frac{1}{\theta^{k-1}(\alpha,z)}.
\endaligned\end{equation}
Here $k-1\geq0$, the desired result follows by Theorem \ref{Th1} and Montgomery's Theorem B.

For $k<1$, by Lemmas \ref{Lemma21} and \ref{Lemma28}, we have the asymptotic
\begin{equation}\aligned\nonumber
\frac{\theta(\beta,z)}{\theta^k(\alpha,z)}\rightarrow\frac{\sqrt{\frac{y}{\beta}}}{(\sqrt{\frac{y}{\alpha}})^k}
=\frac{(\sqrt\alpha)^k}{\sqrt\beta}(\sqrt y)^{1-k}\rightarrow+\infty,\;\;\hbox{as}\;\; y\rightarrow+\infty.
\endaligned\end{equation}
This proves the nonexistence of the maximum.

For Theorem \ref{Coro2}, we follow the steps of proof of Theorem \ref{Th1}. Namely, Theorem \ref{Coro2}
 yielded by Propositions \ref{Prop2.3} and \ref{Prop2.4} in the following.
\begin{proposition}\label{Prop2.3} If $\min_{1\leq j\leq N}\beta_j\geq\max_{1\leq j\leq N}\alpha_j\geq1$ and any $a_j, b_j\geq0$, where $i,j=1\cdots N$ and $N\geq2$ is arbitrary. Then
\begin{equation}\aligned\nonumber
\frac{\partial}{\partial x}\frac{\sum_{j=1}^Nb_j\theta(\beta_j,z)}{\sum_{j=1}^Na_j\theta(\alpha_j,z)}\geq0\;\;\hbox{for}\;\;z\in\mathcal{D}_{\mathcal{G}}.
\endaligned\end{equation}

\end{proposition}

\begin{proof}
It is shown that the derivative of ratio of sum of Theta functions can be decomposed into sum of derivative of ratio of Theta functions. In fact, a direct calculation shows that
\begin{equation}\aligned\nonumber
\partial_x\frac{\sum_{j}b_j\theta(\beta_j,z)}{\sum_{i}a_i\theta(\alpha_i,z)}
&=\sum_{i,j}a_ib_j\frac{
\theta_x(\beta_j,z)\theta(\alpha_i,z)-\theta(\beta_j,z)\theta_x(\alpha_i,z)
}{(\sum_k a_k\theta(\alpha_k,z))^2}
\\
&=\sum_{i,j}a_ib_j\frac{\theta^2(\alpha_i,z)}{(\sum_k a_k\theta(\alpha_k,z))^2}\cdot \partial_x\frac{\theta(\beta_j,z)}{\theta(\alpha_i,z)}.
\endaligned\end{equation}
That is, there exist non-negative functions $c_{ij}$ such that
\begin{equation}\aligned\label{La}
\partial_x\frac{\sum_{j}b_j\theta(\beta_j,z)}{\sum_{i}a_i\theta(\alpha_i,z)}
=\sum_{i,j}c_{ij} \cdot\partial_x\frac{\theta(\beta_j,z)}{\theta(\alpha_i,z)}.
\endaligned\end{equation}
\eqref{La} and Proposition \ref{Prop1A} yield the result.

\end{proof}

\begin{proposition}\label{Prop2.4} If $\min_{1\leq j\leq N}\beta_j\geq\max_{1\leq j\leq N}\alpha_j\geq1$ and any $a_j, b_j\geq0$, where $i,j=1\cdots N$ and $N\geq2$ is arbitrary. Then
\begin{equation}\aligned\nonumber
\frac{\partial}{\partial y}\frac{\sum_{j=1}^Nb_j\theta(\beta_j,z)}{\sum_{j=1}^Na_j\theta(\alpha_j,z)}\mid_{\Re(z)=\frac{1}{2}}\leq0\;\;\hbox{for}\;\;\Im(z)\geq\frac{\sqrt3}{2}.
\endaligned\end{equation}

\end{proposition}

\begin{proof} The idea of the proof is similar to that of Proposition \ref{Prop2.3}. We compute that
\begin{equation}\aligned\nonumber
\partial_y\frac{\sum_{j}b_j\theta(\beta_j,z)}{\sum_{i}a_i\theta(\alpha_i,z)}
&=\sum_{i,j}a_ib_j\frac{
\theta_y(\beta_j,z)\theta(\alpha_i,z)-\theta(\beta_j,z)\theta_y(\alpha_i,z)
}{(\sum_k a_k\theta(\alpha_k,z))^2}
\\
&
=\sum_{i,j}a_ib_j\frac{\theta^2(\alpha_i,z)}{(\sum_k a_k\theta(\alpha_k,z))^2}\cdot \partial_y\frac{\theta(\beta_j,z)}{\theta(\alpha_i,z)}.
\endaligned\end{equation}
Then there exist non-negative functions $c_{ij}$ such that it holds the following kind of linear relation
\begin{equation}\aligned\nonumber
\partial_y\frac{\sum_{j}b_j\theta(\beta_j,z)}{\sum_{i}a_i\theta(\alpha_i,z)}
=\sum_{i,j}c_{ij} \cdot\partial_y\frac{\theta(\beta_j,z)}{\theta(\alpha_i,z)}.
\endaligned\end{equation}

One then restricts the relation on the $\frac{1}{2}-$vertical line,
\begin{equation}\aligned\label{Lb}
\partial_y\frac{\sum_{j}b_j\theta(\beta_j,z)}{\sum_{i}a_i\theta(\alpha_i,z)}\mid_{\Re(z)=\frac{1}{2}}
=\sum_{i,j}c_{ij} \cdot\partial_y\frac{\theta(\beta_j,z)}{\theta(\alpha_i,z)}\mid_{\Re(z)=\frac{1}{2}}.
\endaligned\end{equation}
The sign of $\partial_y\frac{\theta(\beta_j,z)}{\theta(\alpha_i,z)}\mid_{\Re(z)=\frac{1}{2}}$ is non-positive by Proposition \ref{Prop1B}. Then the result follows by \eqref{Lb}.

\end{proof}

\section{Minimum principles and summation formulas}
\setcounter{equation}{0}

There are some nice structures in $\frac{\theta(\beta,z)}{\theta(\alpha,z)}$, as shown in the proof of Theorem \ref{Th1}.
While to prove Theorem \ref{Th2}, we need some minimum principles. In the latter part of this section, we collect some summation formulas and lower, upper-bounds estimates of one-dimensional Theta functions.

\subsection{Minimum princples}
The first minimum principle (inspired by Rankin \cite{Ran1953}) is a baby version of the general ones. It concludes that for any symmetric modular invariant functions satisfying two monotonicity conditions admit the minimum at hexagonal point ($e^{i\frac{\pi}{3}}$).
\begin{proposition}[A minimum principle]\label{Prop31}  Assume that $\mathcal{W}$ is modular invariant, i.e.,
 \begin{equation}\aligned\label{M999}
\mathcal{W}(\frac{az+b}{cz+d})=\mathcal{W}(z),\;\;\hbox{for all}\;\;\left(
                                                                      \begin{array}{cc}
                                                                        a & b \\
                                                                        c & d \\
                                                                      \end{array}
                                                                    \right)\in \hbox{SL}_2(\mathbb{Z}),
\endaligned\end{equation}
and
 \begin{equation}\aligned\nonumber
\mathcal{W}(-\overline{z})=\mathcal{W}(z).
\endaligned\end{equation}

If \begin{equation}\aligned\label{Cabc1}
&\frac{\partial}{\partial y}\mathcal{W}(z)>0,\;\;z=(x,y)\in[0,\frac{1}{2}]\times[a,\infty)\;\;\hbox{for some}\;\;a>\frac{\sqrt3}{2},\\
&\frac{\partial}{\partial x}\mathcal{W}(z)<0,\;\;z=(x,y)\in[0,\frac{1}{2}]\times[b,\infty)\;\;\hbox{for some}\;\;b<\frac{\sqrt3}{2},
\endaligned\end{equation}
and
 \begin{equation}\aligned\label{fff}
\frac{a}{\frac{1}{4}+a^2}\geq b.
\endaligned\end{equation}

Then
\begin{equation}\aligned\nonumber
\min_{z\in\mathbb{H}}\mathcal{W}(z)=\min_{z\in\overline{\mathcal{D}_{\mathcal{G}}}}\mathcal{W}(z)\;\;\hbox{is attained at}\;\; e^{i\frac{\pi}{3}}(\hbox{hexagonal point}).
\endaligned\end{equation}
Here $\mathcal{D}_{\mathcal{G}}$ is the fundamental domain corresponding to modular group $\hbox{SL}_2(\mathbb{Z})$, explicitly, $\mathcal{D}_{\mathcal{G}}=\{
z\in\mathbb{H}: |z|>1,\; 0<x<\frac{1}{2}
\}.$
\end{proposition}

\begin{proof}
By the first part of \eqref{Cabc1}, we have
\begin{equation}\aligned\nonumber
\min_{z\in\overline{\mathcal{D}_{\mathcal{G}}}}\mathcal{W}(z)=\min_{z\in\overline{\mathcal{D}_{\mathcal{G}}}\cap\{y\leq a\}}\mathcal{W}(z).
\endaligned\end{equation}
We then assume $\min_{z\in\overline{\mathcal{D}_{\mathcal{G}}}\cap\{y\leq a\}}\mathcal{W}(z)$ is attained at $z_1:=(x_1,y_1)$.
Then $y_1\leq a$. Since $b<a$, by the second part of \eqref{Cabc1}, we have
\begin{equation}\aligned\label{xxx}
x_1=\frac{1}{2}.
\endaligned\end{equation}
Taking $\left(
                                                                      \begin{array}{cc}
                                                                        0 & 1 \\
                                                                        -1 & 1 \\
                                                                      \end{array}
                                                                    \right)\in \hbox{SL}_2(\mathbb{Z})$,
                                                                    we define
$z_2:=\left(
                                                                      \begin{array}{cc}
                                                                        0 & 1 \\
                                                                        -1 & 1 \\
                                                                      \end{array}
                                                                    \right)z_1$, it follows that
\begin{equation}\aligned\nonumber
z_2=\frac{1}{1-z_1}=\frac{\frac{1}{2}}{\frac{1}{4}+y_1^2}+i \frac{y_1}{\frac{1}{4}+y_1^2}
 \endaligned\end{equation}
and
 \begin{equation}\aligned\nonumber
\mathcal{W}(z_2)=\mathcal{W}(z_1).
 \endaligned\end{equation}
 This implies that $z_2$ still attains the minimum of $\min_{z\in\overline{\mathcal{D}_{\mathcal{G}}}}\mathcal{W}(z)$.
 Now we need an elementary inequality, namely,
  \begin{equation}\aligned\nonumber
\frac{u_2}{\frac{1}{4}+u_2^2}\geq\frac{u_1}{\frac{1}{4}+u_1^2}\;\;\hbox{if}\;\;\frac{1}{2}\leq u_2\leq u_1.
 \endaligned\end{equation}
 From this inequality, one has $\Im(z_2)\geq b$. In fact,
  \begin{equation}\aligned\nonumber
\Im(z_2)=\frac{y_1}{\frac{1}{4}+y_1^2}\geq\frac{a}{\frac{1}{4}+a^2}\;\;\hbox{if}\;\;\frac{1}{2}\leq y_1\leq a.
 \endaligned\end{equation}

By \eqref{fff}, we have $z_2\in\overline{\mathcal{D}_{\mathcal{G}}}\cap\{y\geq b\}$. Still by the second part of \eqref{Cabc1} and $z_2$ is the minimum point, there must be $\Re(z_2)=\frac{1}{2}$, i.e., $\frac{\frac{1}{2}}{\frac{1}{4}+y_1^2}=\frac{1}{2}$. It yields that $y_1=\frac{\sqrt3}{2}$. This and \eqref{xxx} yield the result. These complete the proof.

\end{proof}
In many cases, the monotonicity estimates in \eqref{Cabc1} may not hold for such a large domain (cylinder, $y\geq b$). In fact, we can replace such a large domain (an infinite cylinder) to a finite rectangle domain. While we should add a comparison inequality as
$$\mathcal{W}(z)>\mathcal{W}(z_0) \;\hbox{for some}\; z_0\in\mathcal{D}_{\mathcal{G}}\cap\{y< c\}, \;\hbox{and any}\; z\in\mathcal{D}_{\mathcal{G}}\cap\{y\geq c\}.$$
In practice, such a point $z_0$ can be chosen to a very special and easily calculated point like $i$ or $e^{i\frac{\pi}{3}}$.

We state it precisely for  application as follows.

\begin{proposition}[A refined minimum principle]\label{Prop32}  Assume that $\mathcal{W}$ is modular invariant, i.e.,
 \begin{equation}\aligned\nonumber
\mathcal{W}(\frac{az+b}{cz+d})=\mathcal{W}(z),\;\;\hbox{for all}\;\;\left(
                                                                      \begin{array}{cc}
                                                                        a & b \\
                                                                        c & d \\
                                                                      \end{array}
                                                                    \right)\in \hbox{SL}_2(\mathbb{Z}),
\endaligned\end{equation}
and
 \begin{equation}\aligned\nonumber
\mathcal{W}(-\overline{z})=\mathcal{W}(z).
\endaligned\end{equation}

If
\begin{itemize}
  \item [(1)] $\frac{\partial}{\partial y}\mathcal{W}(z)>0,\;\;z=(x,y)\in[0,\frac{1}{2}]\times[a,c]\;\;\hbox{for some}\;\;a>\frac{\sqrt3}{2}$;
  \item [(2)] $\frac{\partial}{\partial x}\mathcal{W}(z)<0,\;\;z=(x,y)\in[0,\frac{1}{2}]\times[b,c]\;\;\hbox{for some}\;\;b<\frac{\sqrt3}{2}$;
  \item [(3)] $\mathcal{W}(z)>\mathcal{W}(z_0)$ for some $z_0\in\mathcal{D}_{\mathcal{G}}\cap\{y< c\}$, \;{and any}\; $z\in\mathcal{D}_{\mathcal{G}}\cap\{y\geq c\}$ where $c>a$.
\end{itemize}
Here
 \begin{equation}\aligned\nonumber
\frac{a}{\frac{1}{4}+a^2}\geq b
\endaligned\end{equation}
and $\mathcal{D}_{\mathcal{G}}$ is the fundamental domain corresponding to modular group $\hbox{SL}_2(\mathbb{Z})$, explicitly, $\mathcal{D}_{\mathcal{G}}=\{
z\in\mathbb{H}: |z|>1,\; 0<x<\frac{1}{2}
\}.$
Then
\begin{equation}\aligned\nonumber
\min_{z\in\mathbb{H}}\mathcal{W}(z)=\min_{z\in\overline{\mathcal{D}_{\mathcal{G}}}}\mathcal{W}(z)\;\;\hbox{is attained at}\;\; e^{i\frac{\pi}{3}}(\hbox{hexagonal point}).
\endaligned\end{equation}

\end{proposition}

\begin{proof} Item $(3)$ implies that
\begin{equation}\aligned\nonumber
\min_{z\in\overline{\mathcal{D}_{\mathcal{G}}}}\mathcal{W}(z)
=\min_{z\in\overline{\mathcal{D}_{\mathcal{G}}}\cap\{y\leq c\}}\mathcal{W}(z).
\endaligned\end{equation}
The rest of the proof is similar to the proof of Proposition \ref{Prop31}, hence we omit the details here.

\end{proof}

 We shall use Proposition \ref{Prop32} to prove Theorem \ref{Th2}.
 To Proposition \ref{Prop32},
 we shall select suitably of the pair $(a,b)$ satisfying $\frac{a}{\frac{1}{4}+a^2}\geq b$ and $a>\frac{\sqrt3}{2}$.
 It is crucial to select the pair $(a,b)$. In the following, we choose $(a,b)=(\frac{4}{3},\frac{48}{73})$
 and $c=2$. The corresponding estimates of (1), (2) and (3) are established in Lemmas \ref{Lemma3a}-\ref{Lemma3c} respectively.

%%%%%%%%%%%%%%%%%%%%%%%%%%%%%%%%%%%%%%%%%%%%%%%%%%%%%%%%%%%%%%%%%%%%%%%%%%%%%%%%%%%%%%%%%%%%%%%%%%%%
%%%%%%%%%%%%%%%%%%%%%%%%%%%%%%%%%%%%%%%%%%%%%%%%%%%%%%%%%%%%%%%%%%%%%%%%%%%%%%%%%%%%%%%%%%%%%%%%%%%%

We have the following computation at some particular point.

\begin{lemma}\label{Lemma3a} Assume that $s>1, \alpha\geq 2s$. Then

\begin{equation}\aligned\nonumber
\frac{\zeta(s,z)}{\theta^k(\alpha,z)}\mid_{\Im(z)\geq2}
\big/\frac{\zeta(s,z)}{\theta^k(\alpha,z)}\mid_{z=e^{i\frac{\pi}{3}}}>1.
 \endaligned\end{equation}

This implies that for $\alpha\geq s+10, s\geq2$, it holds that
\begin{equation}\aligned\nonumber
\min_{z\in\mathbb{H}}\frac{\zeta(s,z)}{\theta^k(\alpha,z)}
=\min_{z\in\overline{\mathcal{D}_{\mathcal{G}}}}\frac{\zeta(s,z)}{\theta^k(\alpha,z)}
=\min_{z\in\overline{\mathcal{D}_{\mathcal{G}}}\cap\{y\leq2\}}\frac{\zeta(s,z)}{\theta^k(\alpha,z)}.
 \endaligned\end{equation}
\end{lemma}

\begin{lemma}\label{Lemma3b} Assume that $s>1, \alpha\geq 2s$. Then

\begin{equation}\aligned\nonumber
\frac{\partial}{\partial x}\frac{\zeta(s,z)}{\theta^k(\alpha,z)}<0\;\;\hbox{for}\;\;z\in\mathcal{D}_{\mathcal{G}}\cap\{2\geq y\geq \frac{48}{73}\}.
 \endaligned\end{equation}

\end{lemma}

\begin{lemma}\label{Lemma3c} Assume that $s>1, \alpha\geq 2s$. Then

\begin{equation}\aligned\nonumber
\frac{\partial}{\partial y}\frac{\zeta(s,z)}{\theta^k(\alpha,z)}>0\;\;\hbox{for}\;\;z\in\mathcal{D}_{\mathcal{G}}\cap\{2\geq y\geq \frac{4}{3}\}.
 \endaligned\end{equation}

\end{lemma}

\subsection{Summation formulas}
To prove Lemmas \ref{Lemma3a}-\ref{Lemma3c}, we need some preliminary and auxiliary tools.

We first recall some basic estimates. By skilfully using the Euler-Maclaurin summation formula, Rankin deduced that in his paper implicitly
\begin{lemma}[A summation formula, Rankin 1953 \cite{Ran1953}]\label{Rankin1} Assume that $z\in \mathbb{H}$ and $s>1$. Then
\begin{equation}\aligned\nonumber
\sum_{n\in\mathbb{Z}}\frac{1}{|mz+n|^{2s}}=\frac{\Gamma(\frac{1}{2})\Gamma(s-\frac{1}{2})}{\Gamma(s)}y^{1-2s}\frac{1}{m^{2s-1}}
+\sigma\cdot\frac{s}{2}\frac{(2s+1)^{s+\frac{1}{2}}}{(2s+2)^{s+1}}y^{-(2s+1)}\frac{1}{m^{2s+1}},\;\;\sigma\in[-1,1].
 \endaligned\end{equation}
Here $z=x+iy$ is a complex number in the upper half plane.
\end{lemma}
It looks that there is no exact and explicit summation formula for $\sum_{n\in\mathbb{Z}}\frac{1}{|mz+n|^{2s}}$, hence Lemma \ref{Rankin1} is the best available one to use. In Lemma \ref{Rankin1}, one can view that

\begin{equation}\aligned\nonumber
\sum_{n\in\mathbb{Z}}\frac{1}{|mz+n|^{2s}}=\hbox{approximate part}+\hbox{error part}.
 \endaligned\end{equation}

And at least for large $y$, we get

\begin{equation}\aligned\nonumber
&\hbox{approximate part}:=\frac{\Gamma(\frac{1}{2})\Gamma(s-\frac{1}{2})}{\Gamma(s)}y^{1-2s}\frac{1}{m^{2s-1}},\\
&\hbox{error part}:=\sigma\cdot\frac{s}{2}\frac{(2s+1)^{s+\frac{1}{2}}}{(2s+2)^{s+1}}y^{-(2s+1)}\frac{1}{m^{2s+1}}.
 \endaligned\end{equation}

By Lemma \ref{Rankin1}, one has

\begin{lemma}\label{Rankin2} Assume that $z\in \mathbb{H}$ and $s>1$. Then
\begin{equation}\aligned\nonumber
\sum_{m=1}^\infty\sum_{n\in\mathbb{Z}}\frac{1}{|mz+n|^{2s}}=\xi(2s-1)\frac{\Gamma(\frac{1}{2})\Gamma(s-\frac{1}{2})}{\Gamma(s)}y^{1-2s}
+\sigma\cdot\xi(2s+1)\frac{s}{2}\frac{(2s+1)^{s+\frac{1}{2}}}{(2s+2)^{s+1}}y^{-(2s+1)},\;\;\sigma\in[-1,1].
 \endaligned\end{equation}

\end{lemma}

Finally, one obtains the approximate and error part of Zeta functions $\zeta(s,z)$.

\begin{lemma}\label{Rankin3} Assume that $z\in \mathbb{H}$ and $s>1$. Then
 \begin{equation}\aligned\nonumber
\zeta(s,z)&=\sum_{(m,n)\in\mathbb{Z}^2\backslash\{0\}}\frac{y^s}{|mz+n|^{2s}}\\
&=2\xi(2s)y^s+
2\xi(2s-1)\frac{\Gamma(\frac{1}{2})\Gamma(s-\frac{1}{2})}{\Gamma(s)}y^{1-s}
+\sigma\cdot\xi(2s+1)s\frac{(2s+1)^{s+\frac{1}{2}}}{(2s+2)^{s+1}}y^{-(s+1)},\;\;\sigma\in[-1,1].
\endaligned\end{equation}

\end{lemma}

\begin{proof}
Recall that
 \begin{equation}\aligned\label{Lq}
\zeta(s,z)=\sum_{(m,n)\in\mathbb{Z}^2\backslash\{0\}}\frac{y^s}{|mz+n|^{2s}}.
\endaligned\end{equation}

We split the summation in terms of $m$ into $m=0$, $m>0$, $m<0$.
Note that when $m=0$, the double summation in \eqref{Lq} becomes
 \begin{equation}\aligned\nonumber
\sum_{n\in\mathbb{Z}\backslash\{0\}}\frac{y^s}{|n|^{2s}}=2y^s\sum_{n=1}^\infty \frac{1}{n^{2s}}=2\xi(2s)y^s.
\endaligned\end{equation}

Then

\begin{equation}\aligned\label{JJ2}
\zeta(s,z)=\sum_{(m,n)\in\mathbb{Z}^2\backslash\{0\}}\frac{y^s}{|mz+n|^{2s}}
=&2\xi(2s)y^s
+\sum_{m=1}^\infty\sum_{n\in\mathbb{Z}}\frac{y^s}{|mz+n|^{2s}}
+\sum_{m=-\infty}^{-1}\sum_{n\in\mathbb{Z}}\frac{y^s}{|mz+n|^{2s}}\\
=&2\xi(2s)y^s
+2\sum_{m=1}^\infty\sum_{n\in\mathbb{Z}}\frac{y^s}{|mz+n|^{2s}}.
\endaligned\end{equation}

Therefore, by Lemma \ref{Rankin2},

 \begin{equation}\aligned\nonumber
\zeta(s,z)=2\xi(2s)y^s+
2\xi(2s-1)\frac{\Gamma(\frac{1}{2})\Gamma(s-\frac{1}{2})}{\Gamma(s)}y^{1-s}
+\sigma\cdot\xi(2s+1)s\frac{(2s+1)^{s+\frac{1}{2}}}{(2s+2)^{s+1}}y^{-(s+1)},\;\;\sigma\in[-1,1].
\endaligned\end{equation}

\end{proof}

There is an another useful tool. Rankin deduced that in his paper implicitly
\begin{lemma}[A summation formula, Rankin 1953 \cite{Ran1953}]\label{Rankin4} Assume that $z\in \mathbb{H}$ and $s>1$. Then
 \begin{equation}\aligned\nonumber
\zeta_y(s,z)=&2s\Big[
\xi(2s)y^{s-1}-
\frac{s-1}{s}\xi(2s-1)\frac{\Gamma(\frac{1}{2})\Gamma(s-\frac{1}{2})}{\Gamma(s)}y^{-s}\\
&+\xi(2s+1)\big(\sigma_1\cdot\frac{1}{2}s\frac{(2s+1)^{s+\frac{1}{2}}}{(2s+2)^{s+1}}
+\sigma_2\cdot(s+1)\frac{(2s+3)^{s+\frac{3}{2}}}{(2s+4)^{s+2}}
\big)
y^{-(s+2)}\Big],\;\;\sigma_1, \sigma_2\in[-1,1].
\endaligned\end{equation}
\end{lemma}

\begin{proof}
Recall that
 \begin{equation}\aligned\nonumber
\zeta(s,z)=\sum_{(m,n)\in\mathbb{Z}^2\backslash\{0\}}\frac{y^s}{|mz+n|^{2s}}.
\endaligned\end{equation}
A direct calculation yields that
 \begin{equation}\aligned\nonumber
\zeta_y(s,z)&=\sum_{(m,n)\in\mathbb{Z}^2\backslash\{0\}}\frac{sy^{s-1}}{|mz+n|^{2s}}-\frac{2sy^{s+1}m^2}{|mz+n|^{2(s+1)}}.
\endaligned\end{equation}
Splitting the summation in terms of $m$ into $m=0$, $m>0$, $m<0$, and by symmetry,

 \begin{equation}\aligned\nonumber
\zeta_y(s,z)&=2s\xi(2s)y^{s-1}+2sy^{s-1}\sum_{m=1}^\infty\sum_{n\in\mathbb{Z}}\frac{1}{|mz+n|^{2s}}
-4sy^{s+1}\sum_{m=1}^\infty m^2\sum_{n\in\mathbb{Z}}\frac{1}{|mz+n|^{2(s+1)}}.
\endaligned\end{equation}
Note that the summation $\sum_{n\in\mathbb{Z}}\frac{1}{|mz+n|^{2s}}$ is studied in Lemma \ref{Rankin1}. The rest of the proof
 followed by Lemmas \ref{Rankin1} and \ref{Rankin2}.

\end{proof}
At the end of this section, we recall some estimates on one-dimensional Theta functions.
Recall that in \eqref{TXY}
\begin{equation}\aligned\nonumber
\vartheta(X;Y):=\sum_{n\in\mathbb{Z}} e^{-\pi n^2 X} e^{2n\pi i Y},
 \endaligned\end{equation}
 where $X>0$ and $Y\in\R$.

The following Lemmas \ref{LemmaT1} and \ref{LemmaT2} are proved in \cite{LW2022}.
\begin{lemma}\label{LemmaT1}\cite{LW2022}. Assume $X>\frac{1}{5}$. If $\sin(2\pi Y)>0$, then
\begin{equation}\aligned\nonumber
-\overline\vartheta(X)\sin(2\pi Y)\leq\frac{\partial}{\partial Y}\vartheta(X;Y)\leq-\underline\vartheta(X)\sin(2\pi Y).
 \endaligned\end{equation}
If $\sin(2\pi Y)<0$, then
\begin{equation}\aligned\nonumber
-\underline\vartheta(X)\sin(2\pi Y)\leq\frac{\partial}{\partial Y}\vartheta(X;Y)\leq-\overline\vartheta(X)\sin(2\pi Y).
 \endaligned\end{equation}
Here
\begin{equation}\aligned\nonumber
\underline\vartheta(X):=4\pi e^{-\pi X}(1-\mu(X)), \;\; \overline\vartheta(X):=4\pi e^{-\pi X}(1+\mu(X)),
 \endaligned\end{equation}
and
\begin{equation}\label{mmmx}
\mu(X):=\sum_{n=2}^\infty n^2 e^{-\pi(n^2-1)X}.
\end{equation}

\end{lemma}

\begin{lemma}\label{LemmaT2}\cite{LW2022}.
Assume $X<\min\{\frac{\pi}{\pi+2},\frac{\pi}{4\log\pi}\}=\frac{\pi}{\pi+2}$. If $\sin(2\pi Y)>0$, then
\begin{equation}\aligned\nonumber
-\overline\vartheta(X)\sin(2\pi Y)\leq\frac{\partial}{\partial Y}\vartheta(X;Y)\leq-\underline\vartheta(X)\sin(2\pi Y).
 \endaligned\end{equation}
If $\sin(2\pi Y)<0$, then
\begin{equation}\aligned\nonumber
-\underline\vartheta(X)\sin(2\pi Y)\leq\frac{\partial}{\partial Y}\vartheta(X;Y)\leq-\overline\vartheta(X)\sin(2\pi Y).
 \endaligned\end{equation}
Here
\begin{equation}\aligned\nonumber
\underline\vartheta(X):=\pi e^{-\frac{\pi}{4X}}X^{-\frac{3}{2}};\;\; \overline\vartheta(X):=X^{-\frac{3}{2}}.
 \endaligned\end{equation}

\end{lemma}

\section{Proof of Theorem \ref{Th2} and its corollary}
\setcounter{equation}{0}

In this section, we give the proof of Theorem \ref{Th2}. By the minimum principle given by Proposition \ref{Prop32}, it suffices to prove Lemmas \ref{Lemma3a}-\ref{Lemma3c}. We prove Lemmas \ref{Lemma3b}, \ref{Lemma3c}, and \ref{Lemma3a} in Subsections 4.1, 4.2 and 4.3 respectively. In Subsection 4.4, we give the proof of Corollary \ref{Coro1.3}.

\subsection{$\partial_x$ estimates}
By a direct computation and deformation,
\begin{equation}\aligned\nonumber
\partial_x\frac{\zeta(s,z)}{\theta^k(\alpha,z)}
&=\frac{\zeta(s,z)}{\theta^k(\alpha,z)}\cdot
\Big(
\frac{\zeta_x(s,z)}{\zeta(s,z)}-k\frac{\theta_x(\alpha,z)}{\theta(\alpha,z)}
\Big)\\
&=-\frac{\zeta(s,z)}{\theta^k(\alpha,z)}\sin(2\pi x)\cdot
\Big(
\frac{-\zeta_x(s,z)}{\sin(2\pi x)\zeta(s,z)}-k\frac{-\theta_x(\alpha,z)}{\sin(2\pi x)\theta(\alpha,z)}
\Big).
\endaligned\end{equation}

Lemma \ref{Lemma3b} is equivalent to

\begin{lemma}\label{Lemma3b1} Assume that $\alpha\geq 3s, s>1$. Then
\begin{equation}\aligned\nonumber
\frac{-\zeta_x(s,z)}{\sin(2\pi x)\zeta(s,z)}-2s\frac{-\theta_x(\alpha,z)}{\sin(2\pi x)\theta(\alpha,z)}>0\;\;\hbox{for}\;\;z\in\mathcal{D}_{\mathcal{G}}\cap\{2\geq y\geq \frac{48}{73}\}.
 \endaligned\end{equation}

\end{lemma}

In the rest of this subsection, we prove Lemma \ref{Lemma3b1}.
To prove it, we estimate $\frac{-\zeta_x(s,z)}{\sin(2\pi x)\zeta(s,z)}$
and $\frac{-\theta_x(s,z)}{\sin(2\pi x)\theta(s,z)}$ separately.
For $\frac{-\theta_x(s,z)}{\sin(2\pi x)\theta(s,z)}$, we use lower and upper bounds given by Lemmas \ref{LemmaT1} and \ref{LemmaT2}.
While for $\frac{-\zeta_x(s,z)}{\sin(2\pi x)\zeta(s,z)}$, getting a lower bound of it directly becomes complicated.
To overcome it, we use a relation between $\zeta(s,z)$ and $\theta(\alpha,z)$. Namely, we use the identity
 \begin{equation}\aligned\nonumber
\zeta(s,z)=\frac{\pi^s}{\Gamma(s)}\int_0^\infty \big(\theta(\alpha,z)-1\big)\alpha^{s-1}d\alpha.
\endaligned\end{equation}

Then by taking derivative with respect to $x$, we get an identity of $\zeta_x(s,z)$ in terms of $\theta_x(\alpha,z)$

 \begin{equation}\aligned\label{Zetax}
\zeta_x(s,z)=\frac{\pi^s}{\Gamma(s)}\int_0^\infty \theta_x(\alpha,z)\alpha^{s-1}d\alpha.
\endaligned\end{equation}

Using \eqref{Zetax}, we can get the bounds of $\zeta_x(s,z)$ by bounds of $\theta_x(\alpha,z)$. Together with summation formula
\eqref{Rankin3}, we can bound $\frac{-\zeta_x(s,z)}{\sin(2\pi x)\zeta(s,z)}$.
%%%%%%%%%%%%%%%%%%%%%%%%%%%

We now start the detailed proof.
With the expression of theta function in Lemma \ref{Lemma21}, namely,
 \begin{equation}\aligned\label{CM300}
\theta(\alpha,z)=\sqrt{\frac{y}{\alpha}}\cdot\sum_{n\in\mathbb{Z}} e^{-\pi\alpha yn^2}\vartheta(\frac{y}{\alpha};nx)
.
\endaligned\end{equation}

Using bounds of 1-d theta functions ($\vartheta(X;Y)$) given by Lemmas \ref{LemmaT1} and \ref{LemmaT2}, one has

\begin{lemma}[An upper bound of $\frac{-\theta_x(\alpha,z)}{\sin(2\pi x)}$]\label{Lemma4.2}

Depending on the value of $\frac{y}{\alpha}$, it holds that

\begin{itemize}
  \item for $\frac{y}{\alpha}\geq\frac{1}{5}$, $$\frac{-\theta_x(\alpha,z)}{\sin(2\pi x)}\leq
8\pi(1+\mu(\frac{y}{\alpha}))\sqrt{\frac{y}{\alpha}}\sum_{n=1}^\infty n^2 e^{-\pi y (n^2 \alpha+\frac{1}{\alpha})};$$
  \item for $\frac{y}{\alpha}\leq\frac{\pi}{\pi+2}$, $$\frac{-\theta_x(\alpha,z)}{\sin(2\pi x)}\leq2(\frac{y}{\alpha})^{-1}\sum_{n=1}^\infty n^2 e^{-\pi\alpha y n^2}.$$
\end{itemize}

\end{lemma}

With Lemma \ref{Lemma4.2}, to bound $\frac{-\theta_x(s,z)}{\sin(2\pi x)\theta(s,z)}$, we need a lower bound of $\theta(\alpha,z)$. By
\eqref{CM300}, we have
\begin{lemma}[Lower bounds of $\theta(\alpha,z)$]\label{Lemma4.3} Assume that $\alpha, y>0$.
It holds that
\begin{itemize}
  \item  for $\frac{y}{\alpha}\geq1$, then $\theta(\alpha,z)\geq\sqrt\frac{y}{\alpha}$.
  \item  for $\frac{y}{\alpha}\leq1$, then $\theta(\alpha,z)\geq1$.
\end{itemize}
\end{lemma}

\begin{proof} The first part is trivial. The second part based on a duality formula of Jacobi theta function of third type. Recall that $
\vartheta_3(x)=\sum_{n\in\mathbb{Z}}e^{-\pi n^2 x}.$
Then $\theta(\alpha,z)\geq\sqrt{\frac{y}{\alpha}}\vartheta(\frac{y}{\alpha};0)
=\sqrt{\frac{y}{\alpha}}\vartheta_3(\frac{y}{\alpha})=\vartheta_3(\frac{\alpha}{y})\geq1$.

\end{proof}

Combining Lemma \ref{Lemma4.2} with Lemma \ref{Lemma4.3}, we get an upper bound of $\frac{-\theta_x(\alpha,z)}{\sin(2\pi x)\theta(\alpha,z)}$.
\begin{lemma}[An upper bound of $\frac{-\theta_x(\alpha,z)}{\sin(2\pi x)\theta(\alpha,z)}$]
\label{Lemma44KKK}
Depending on the value of $\frac{y}{\alpha}$, it holds that

\begin{itemize}
  \item for $\frac{y}{\alpha}\geq\frac{1}{5}$, $$\frac{-\theta_x(\alpha,z)}{\sin(2\pi x)\theta(\alpha,z)}\leq
8\pi(1+\mu(\frac{y}{\alpha}))\sum_{n=1}^\infty n^2 e^{-\pi y (n^2 \alpha+\frac{1}{\alpha})};$$
  \item for $\frac{y}{\alpha}\leq\frac{\pi}{\pi+2}$, $$\frac{-\theta_x(\alpha,z)}{\sin(2\pi x)\theta(\alpha,z)}\leq2(\frac{y}{\alpha})^{-\frac{3}{2}}\sum_{n=1}^\infty n^2 e^{-\pi\alpha y n^2}.$$
\end{itemize}

\end{lemma}

We proceed to get the lower bound of $\frac{-\zeta_x(\alpha,z)}{\sin(2\pi x)\zeta(\alpha,z)}$. Using \eqref{Zetax}, we first estimate the lower bound of $\theta_x(\alpha,z)$. By \eqref{CM300} and bounds of 1-d theta functions ($\vartheta(X;Y)$) given by Lemmas \ref{LemmaT1} and \ref{LemmaT2}, we have
\begin{lemma}[A lower bound of $\frac{-\theta_x(\alpha,z)}{\sin(2\pi x)}$]\label{Lemma43}

Depending on the value of $\frac{y}{\alpha}$, it holds that

\begin{itemize}
  \item for $\frac{y}{\alpha}\geq\frac{1}{5}$, $$\frac{-\theta_x(\alpha,z)}{\sin(2\pi x)}\geq
8\pi(1-\mu(\frac{y}{\alpha}))\sqrt{\frac{y}{\alpha}}\sum_{n=1}^\infty n^2 e^{-\pi y (n^2 \alpha+\frac{1}{\alpha})};$$
  \item for $\frac{y}{\alpha}\leq\frac{\pi}{\pi+2}$, $$\frac{-\theta_x(\alpha,z)}{\sin(2\pi x)}\geq2\pi(\frac{y}{\alpha})^{-1}\sum_{n=1}^\infty n^2 e^{-\pi (n^2 y +\frac{1}{4y})\alpha}.$$
\end{itemize}

Here
 \begin{equation}\aligned\nonumber
\mu(X):=\sum_{n=2}^\infty n^2 e^{-\pi (n^2-1)X}.
\endaligned\end{equation}
\end{lemma}

Proceeding by \eqref{Zetax},
 \begin{equation}\aligned\label{Zetaxx}
\frac{-\zeta_x(s,z)}{\sin(2\pi x)}&=\frac{\pi^s}{\Gamma(s)}\int_0^\infty \frac{-\theta_x(\alpha,z)}{\sin(2\pi x)}\alpha^{s-1}d\alpha\\
&=\frac{\pi^s}{\Gamma(s)}\Big(\int_0^y \frac{-\theta_x(\alpha,z)}{\sin(2\pi x)}\alpha^{s-1}d\alpha
+\int_y^\infty \frac{-\theta_x(\alpha,z)}{\sin(2\pi x)}\alpha^{s-1}d\alpha\Big)\\
&\geq\frac{\pi^s}{\Gamma(s)}\int_y^\infty \frac{-\theta_x(\alpha,z)}{\sin(2\pi x)}\alpha^{s-1}d\alpha.
\endaligned\end{equation}
Using the incomplete gamma function $\Gamma(s,x)$, which is defined as
 \begin{equation}\aligned\nonumber
\Gamma(s,x)=\int_{x}^\infty t^{s-1}e^{-t}dt,
\endaligned\end{equation}
together with Lemma \ref{Lemma43} and \eqref{Zetaxx}, we have
\begin{lemma}[A lower bound of $\frac{-\zeta_x(s,z)}{\sin(2\pi x)}$]\label{Lemma4.6}
For $s>1, y>0$, it holds that
 \begin{equation}\aligned\nonumber
\frac{-\zeta_x(s,z)}{\sin(2\pi x)}\geq\frac{\pi^s}{\Gamma(s)}\frac{2\pi}{y}
\sum_{n=1}^\infty  n^2\cdot\frac{\Gamma(s+1,\pi (n^2y^2+\frac{1}{4}))}{(\pi (n^2 y+\frac{1}{4y}))^{s+1}}.
\endaligned\end{equation}

\end{lemma}

Now we need a lower bound of the incomplete gamma function $\Gamma(s,x)$.

Integrating by parts, one has the recursion in $s$, i.e.,
 \begin{equation}\aligned\label{Rec}
\Gamma(s,x)=x^{s-1}e^{-x}+(s-1)\Gamma(s-1,x)\;\;\hbox{for}\;\;s\geq2.
\endaligned\end{equation}

We need a lower bound for incomplete gamma function $\Gamma(s,x)$. By using the recursion formula given by \eqref{Rec} and some monotonicity properties, Pinelis \cite{Pin2020} deduced that
\begin{lemma}[Lower-bound functions of incomplete gamma function $\Gamma(s,x)$]\label{Lemma4.7}

The incomplete gamma function $\Gamma(s,x)$ has the following

 \begin{equation}\aligned\nonumber
\Gamma(s,x)
\begin{cases}
>\Big(
\frac{(x+2)^s-x^s-2^s}{2s}+\Gamma(s)
\Big)e^{-x}\;\;&\hbox{for}\;\; s>3;\\
=(x^2+2x+2)e^{-x}\;\;&\hbox{for}\;\; s=3;\\
>\Big(
\frac{(x+2)^{s-1}+x^{s-1}-2^{s-1}}{2}+\Gamma(s)
\Big)e^{-x}\;\;&\hbox{for}\;\; s\in(2,3);\\
=(x+1)e^{-x}\;\;&\hbox{for}\;\; s=2;\\
>\Big(
\frac{(x+2)^s-x^s-2^s}{2s}+\Gamma(s)
\Big)e^{-x}\;\;&\hbox{for}\;\; s\in(1,2);\\
=e^{-x}\;\;&\hbox{for}\;\; s=1.\\
\end{cases}
\endaligned\end{equation}

\end{lemma}

Using Lemmas \ref{Lemma4.7} and \ref{Lemma4.6}, we have

\begin{lemma}[A lower bound of $\frac{-\zeta_x(s,z)}{\sin(2\pi x)}$]\label{Lemma4.8}
For $s>1, x\in[0,\frac{1}{2}]$ and $y>0$, it holds that
 \begin{equation}\aligned\nonumber
\frac{-\zeta_x(s,z)}{\sin(2\pi x)}\geq\frac{2s}{y}
\big(
 y+\frac{1}{4y}
\big)^{-(s+1)} e^{-\pi(y^2+\frac{1}{4})}.
\endaligned\end{equation}

\end{lemma}
Now we are ready to obtain an effective lower bound of $\frac{-\zeta_x(s,z)}{\sin(2\pi x)\zeta(s,z)}$.
By Lemma \ref{Lemma4.8} and the upper bound of $\zeta(s,z)$ given by the summation formula in Lemma \ref{Rankin3}, we obtain that

\begin{lemma}[A lower bound of $\frac{-\zeta_x(s,z)}{\sin(2\pi x)\zeta(s,z)}$]\label{Lemma4.9}

For $s>1, x\in[0,\frac{1}{2}]$ and $y>0$, it holds that

 \begin{equation}\aligned\nonumber
\frac{-\zeta_x(s,z)}{\sin(2\pi x)\zeta(s,z)}
\geq\frac{\frac{2s}{y}
\big(
 y+\frac{1}{4y}
\big)^{-(s+1)} e^{-\pi(y^2+\frac{1}{4})}}
{2\xi(2s)y^s+
2\xi(2s-1)\frac{\Gamma(\frac{1}{2})\Gamma(s-\frac{1}{2})}{\Gamma(s)}y^{1-s}
+\xi(2s+1)s\frac{(2s+1)^{s+\frac{1}{2}}}{(2s+2)^{s+1}}y^{-(s+1)}}.
\endaligned\end{equation}

\end{lemma}
 Lemma \ref{Lemma4.9} is quite useful when the parameter $s$ is large, while when $s$ is small, we have a more precise bound (Lemma \ref{Lemma4.11}).
By Lemmas \ref{Lemma44KKK} and \ref{Lemma4.9}, to prove the main Lemma \ref{Lemma3b1} for the cases $s\geq4$, it suffices to prove that
%%%%%%%%%%%%%%%%%%%%%%%%%%%%%%%%%%%%%%%%%%%%%%%%%%%%%%%%%%%
\begin{lemma}[An elementary inequality]\label{Lemma4.10} Assume that $s\geq4, \alpha\geq3s$. Then for $y\in[\frac{48}{73},2]$, it holds that
 \begin{equation}\aligned\nonumber
&\frac{\frac{2s}{y}
\big(
 y+\frac{1}{4y}
\big)^{-(s+1)} e^{-\pi(y^2+\frac{1}{4})}}
{2\xi(2s)y^s+
2\xi(2s-1)\frac{\Gamma(\frac{1}{2})\Gamma(s-\frac{1}{2})}{\Gamma(s)}y^{1-s}
+\xi(2s+1)s\frac{(2s+1)^{s+\frac{1}{2}}}{(2s+2)^{s+1}}y^{-(s+1)}}\\
\geq&\begin{cases}
3(\frac{y}{\alpha})^{-\frac{3}{2}}e^{-\pi\alpha y}\;&\hbox{for}\;\frac{y}{\alpha}\leq1,\\
9\pi e^{-\pi y ( \alpha+\frac{1}{\alpha})}\;&\hbox{for}\;\frac{y}{\alpha}\geq1.
\end{cases}
\endaligned\end{equation}

\end{lemma}
Since $\alpha\geq 3s\geq12$, the proof of Lemma \ref{Lemma4.10} is trivial, hence we omit the details here.

It remains to prove the main Lemma \ref{Lemma3b1} for the cases $s\in(1,4]$. When $s$ is small, we could deduce a
precise lower bound for $\frac{-\zeta_x(s,z)}{\zeta(s,z)\sin(2\pi x)}$. We first have
\begin{lemma}[A lower bound of $\frac{-\zeta_x(s,z)}{\sin(2\pi x)}: s\leq4$]\label{Lemma4.11}
Assume that $s\in(1,4], x\in[0,\frac{1}{2}]$. Then for $y\geq\frac{48}{73}$, it holds that
 \begin{equation}\aligned\nonumber
\frac{-\zeta_x(s,z)}{\sin(2\pi x)}\geq\frac{16}{3}\sqrt{y}\frac{\pi^{s+1}}{\Gamma(s)}K_{s-\frac{1}{2}}(2\pi y)
.
\endaligned\end{equation}
Here $K_s(z)$ is the modified Bessel function of the second kind and is defined as
 \begin{equation}\aligned\nonumber
K_{s}(y)=\frac{1}{2}\int_0^\infty  t^{-(s+1)}e^{-\frac{1}{2}y(t+\frac{1}{t})}dt,
\endaligned\end{equation}
or
 \begin{equation}\aligned\label{Ks}
K_{s}(y)=\int_0^\infty  e^{-y\cosh(t)}\cosh(st)dt.
\endaligned\end{equation}
\end{lemma}
See more details for $K_s(y)$ in Watson \cite{Watson1995}. To keep the structure clear,
we postpone the proof of Lemma \ref{Lemma4.11} to the end of this subsection.

By Lemma \ref{Lemma4.11} and the upper bound of $\zeta(s,z)$ given by the summation formula in Lemma \ref{Rankin2}, we obtain that
\begin{lemma}[A lower bound of $\frac{-\zeta_x(s,z)}{\zeta(s,z)\sin(2\pi x)}: s\leq4$] \label{Lemma412KKK}
Assume that $s\in(1,4], x\in[0,\frac{1}{2}]$. Then for $y\geq\frac{48}{73}$, it holds that
 \begin{equation}\aligned\nonumber
\frac{-\zeta_x(s,z)}{\zeta(s,z)\sin(2\pi x)}\geq\frac{\frac{16}{3}\sqrt{y}\frac{\pi^{s+1}}{\Gamma(s)}K_{s-\frac{1}{2}}(2\pi y)}
{2\xi(2s)y^s+
2\xi(2s-1)\frac{\Gamma(\frac{1}{2})\Gamma(s-\frac{1}{2})}{\Gamma(s)}y^{1-s}
+\xi(2s+1)s\frac{(2s+1)^{s+\frac{1}{2}}}{(2s+2)^{s+1}}y^{-(s+1)}}.
\endaligned\end{equation}

\end{lemma}

By Lemmas \ref{Lemma44KKK} and \ref{Lemma412KKK}, to prove the main Lemma \ref{Lemma3b1} for the case $s\in(1,4]$, it suffices to prove that
\begin{lemma}[An elementary inequality: (b)]\label{Lemma4.13} Assume that $s\in(1,4], \alpha\geq3s$. Then for $y\in[\frac{48}{73},2]$, it holds that
 \begin{equation}\aligned\nonumber
\frac{\frac{16}{3}\sqrt{y}\frac{\pi^{s+1}}{\Gamma(s)}K_{s-\frac{1}{2}}(2\pi y)}
{2\xi(2s)y^s+
2\xi(2s-1)\frac{\Gamma(\frac{1}{2})\Gamma(s-\frac{1}{2})}{\Gamma(s)}y^{1-s}
+\xi(2s+1)s\frac{(2s+1)^{s+\frac{1}{2}}}{(2s+2)^{s+1}}y^{-(s+1)}}\geq
3(\frac{y}{\alpha})^{-\frac{3}{2}}e^{-\pi\alpha y}.
\endaligned\end{equation}

\end{lemma}
Note that $K_{\frac{1}{2}}(2\pi y)=\frac{2}{\sqrt y}e^{-2\pi y}$ and $K_s(2\pi y)$ is increasing with respect to $s$ by the expression
given by \eqref{Ks}. The proof of Lemma \ref{Lemma4.13} is straightforward and elementary, we then omit the details here.

%%%%%%%%%%%%%%%%%%%%%%%%%%%%%%%%%%%%%%%%%%%%%%%%%%%%%%%%%%%%%%%%%

We now give the proof of Lemma \ref{Lemma4.11}. We use the Chowla-Selberg formula \cite{Chowla1949,Chowla1967}.

\begin{lemma}[Chowla-Selberg formula]\label{Lemma4.14} For $s>1, y>0$, it holds that
 \begin{equation}\aligned\nonumber
\zeta(s,z)&=a_0(s,y)+2\sum_{n=1}^\infty a_n(s,y)\cos(2\pi nx),
\endaligned\end{equation}
where
 \begin{equation}\aligned\label{Aa0}
a_0(s,y):&=2\xi(2s)y^s+
2\xi(2s-1)\frac{\Gamma(\frac{1}{2})\Gamma(s-\frac{1}{2})}{\Gamma(s)}y^{1-s},\\
a_n(s,y):&=\frac{4\pi^s\sqrt{y}}{\Gamma(s)}n^{s-\frac{1}{2}}\sigma_{1-2s}(n)K_{s-\frac{1}{2}}(2\pi ny).
\endaligned\end{equation}
Here $\sigma_{1-2s}(n)=\sum_{d\mid n} d^{1-2s}$, and $K_s(y)$ is the modified Bessel function.

\end{lemma}

By Chowla-Selberg formula (Lemma \ref{Lemma4.14}), using $|\frac{\sin(2\pi nx)}{\sin(2\pi x)}|\leq n$ for $n\in \mathbb{Z}^+$ and $x\in\R$, one has
 \begin{equation}\aligned\label{Zetax1}
\frac{-\zeta_x(s,z)}{4\pi a_1(s,y)\sin(2\pi x)}&=1+\sum_{n=2}^\infty n\frac{a_n(s,y)}{a_1(s,y)}\frac{\sin(2\pi nx)}{\sin(2\pi x)}\\
&\geq1-\sum_{n=2}^\infty n^2\frac{a_n(s,y)}{a_1(s,y)}.
\endaligned\end{equation}
To prove Lemma \ref{Lemma4.11}, by \eqref{Zetax1}, it suffices to prove that

\begin{equation}\aligned\label{Zetax2}
\sum_{n=2}^\infty n^2\frac{a_n(s,y)}{a_1(s,y)}\leq\frac{2}{3}\;\;\hbox{for}\;\; s\in(1,4],\;\; y\geq\frac{48}{73}.
\endaligned\end{equation}
By \eqref{Aa0}, \eqref{Zetax2} is equivalent to
\begin{equation}\aligned\label{Zetax3}
\sum_{n=2}^\infty n^{s+\frac{3}{2}}\sigma_{1-2s}(n)\frac{K_{s-\frac{1}{2}}(2\pi ny)}{K_{s-\frac{1}{2}}(2\pi y)}
\leq\frac{2}{3}\;\;\hbox{for}\;\; s\in(1,4],\;\; y\geq\frac{48}{73}.
\endaligned\end{equation}

Now we need an estimate on ratio of modified Bessel function $K_s(y)$. This is done by the following
\begin{lemma}[Baricz \cite{Bari2010}]\label{Lemma4.15} If $\nu>\frac{1}{2}$. Then for $y>x>0$, it holds that

 \begin{equation}\aligned\nonumber
\frac{K_\nu(y)}{K_\nu(x)}< e^{-(y-x)}(\frac{y}{x})^{-\frac{1}{2}}.
\endaligned\end{equation}

\end{lemma}

By Lemma \ref{Lemma4.15}, to prove \eqref{Zetax3}, it suffices to prove that
\begin{equation}\aligned\label{Zetax4}
\sum_{n=2}^\infty n^{s+1}\sigma_{1-2s}(n)e^{-2\pi(n-1)y}
\leq\frac{2}{3}\;\;\hbox{for}\;\; s\in(1,4],\;\; y\geq\frac{48}{73}.
\endaligned\end{equation}
The proof of \eqref{Zetax4} is elementary, hence we omit the detail here. The proof is complete.

\subsection{$\partial_y$ estimates}

By a direct calculation
\begin{equation}\aligned\nonumber
\partial_y\frac{\zeta(s,z)}{\theta^k(\alpha,z)}
&=\frac{\zeta(s,z)}{\theta^k(\alpha,z)}\cdot
\Big(
\frac{\zeta_y(s,z)}{\zeta(s,z)}-k\frac{\theta_y(\alpha,z)}{\theta(\alpha,z)}
\Big).
\endaligned\end{equation}

Lemma \ref{Lemma3c} is equivalent to

\begin{lemma}\label{Lemma4.16} Assume that $s>1, \alpha\geq 3s$. Then

\begin{equation}\aligned\nonumber
\frac{\zeta_y(s,z)}{\zeta(s,z)}-2s\frac{\theta_y(\alpha,z)}{\theta(\alpha,z)}>0\;\;\hbox{for}\;\;z\in\mathcal{D}_{\mathcal{G}}\cap\{2\geq y\geq \frac{4}{3}\}.
 \endaligned\end{equation}
\end{lemma}

To prove Lemma \ref{Lemma4.16}, we estimate $\frac{\zeta_y(s,z)}{\zeta(s,z)}$ and $\frac{\theta_y(\alpha,z)}{\theta(\alpha,z)}$
respectively. For $\frac{\zeta_y(s,z)}{\zeta(s,z)}$, we use the summation formulas in Lemmas \ref{Rankin3} and \ref{Rankin4}.
For $\frac{\theta_y(\alpha,z)}{\theta(\alpha,z)}$, we deduce by a careful study of properties of $\theta(\alpha,z)$.

By Lemmas \ref{Rankin3} and \ref{Rankin4}, we have
\begin{lemma}[A lower bound of $\frac{\zeta_y(s,z)}{\zeta(s,z)}$]\label{Lemma4.17} Assume that $s>1$. Then

\begin{equation}\aligned\nonumber
\frac{\zeta_y(s,z)}{\zeta(s,z)}\geq\frac{s}{y}
\Big(
1-
\mathcal{A}_s(y)\Big)
.
 \endaligned\end{equation}
Here $\mathcal{A}_s(y)$ is small and explicitly,
\begin{equation}\aligned\label{Asy}
\mathcal{A}_s(y):=\frac{\frac{2s-1}{s}\xi(2s-1)\frac{\Gamma(\frac{1}{2})\Gamma(s-\frac{1}{2})}{\Gamma(s)}y^{1-s}
+\xi(2s+1)\big(
\frac{3}{2}s\frac{(2s+1)^{s+\frac{1}{2}}}{(2s+2)^{s+1}}+(s+1)\frac{(2s+3)^{s+\frac{3}{2}}}{(2s+4)^{s+2}}
\big)y^{-(s+1)}
}
{\xi(2s)y^s+
\xi(2s-1)\frac{\Gamma(\frac{1}{2})\Gamma(s-\frac{1}{2})}{\Gamma(s)}y^{1-s}
+\frac{1}{2}\xi(2s+1)s\frac{(2s+1)^{s+\frac{1}{2}}}{(2s+2)^{s+1}}y^{-(s+1)}}
.
 \endaligned\end{equation}

%%%%%%%%%%%%%%%%%%%%%%%%%%%%%%%%%%%%%%%%
{\bf
\begin{table}[!htbp]\label{TableB}
\caption{{\bf Evaluation of $\mathcal{A}_s(\frac{4}{3})$[Taking six digital numbers].}}
\label{VP}
\centering
\begin{tabular}{|c|c|c|c|c|c|c|c|}

\hline

 $s=1$  & $s=2$ & $s=3$ & $s=4$ & $s=5$& $s=6$& $s=7$& $s=8$\\

\hline

0.886729 & 0.772190 & 0.517878 &0.324054 &0.194742 &0.114367 &0.066316 &0.038192  \\

\hline

\end{tabular}
\end{table}
}

\end{lemma}
For the function $\mathcal{A}_s(y)$ appeared in the lower bound of $\frac{\zeta_y(s,z)}{\zeta(s,z)}$, we have the following
basic properties, whose proof is elementary hence we omit the details here.

\begin{lemma}

\begin{itemize}
  \item For $s\geq1$, $\mathcal{A}_s(y)$ is decreasing for $y\geq1$.

  \item $\max_{y\geq a} \mathcal{A}_s(y)=\mathcal{A}_s(a)$ for $a\geq1$.
\end{itemize}

\end{lemma}

For the upper bound of $\frac{\theta_y(\alpha,z)}{\theta(\alpha,z)}$, we have

\begin{lemma}[An upper bound of $\frac{\theta_y(\alpha,z)}{\theta(\alpha,z)}$]\label{Lemma4.19} Assume that $\alpha\geq 1$. Then

\begin{itemize}
  \item [(1)] For $\frac{y}{\alpha}\geq1$,
  \begin{equation}\aligned\nonumber
\frac{\theta_y(\alpha,z)}{\theta(\alpha,z)}\leq \frac{1}{2y}.
 \endaligned\end{equation}

  \item [(2)] For $\frac{y}{\alpha}\leq1$,
  \begin{equation}\aligned\nonumber
\frac{\theta_y(\alpha,z)}{\theta(\alpha,z)}\leq \frac{2\pi\alpha}{y^2}e^{-\pi\frac{\alpha}{y}}.
 \endaligned\end{equation}

\end{itemize}

\end{lemma}

We postpone the proof of Lemma \ref{Lemma4.19} to the end of this subsection and give the proof of Lemma \ref{Lemma4.16}.
By Lemmas \ref{Lemma4.17} and \ref{Lemma4.19}, to prove Lemma \ref{Lemma4.16}, it suffices to the following
\begin{lemma}\label{Lemma4.20}
Assume that $s\geq1$. Then
\begin{equation}\aligned\nonumber
1-
\mathcal{A}_s(\frac{4}{3})\geq4\pi se^{-\pi s}
.
 \endaligned\end{equation}
 Here $\mathcal{A}_s(y)$ is defined in \eqref{Asy}.
\end{lemma}
Now the proof of Lemma \ref{Lemma4.20} is elementary hence we omit the details here.

It remains to prove Lemma \ref{Lemma4.19}.

\begin{proof} Recall that $\theta(\alpha,z)=\sqrt{\frac{y}{\alpha}}\cdot\sum_{n\in\mathbb{Z}} e^{-\pi\alpha yn^2}\vartheta(\frac{y}{\alpha};nx)$. By a direct calculation
 \begin{equation}\aligned\label{P0}
\theta_y(\alpha,z)=&\frac{1}{2\sqrt{\alpha y}}\sum_{n\in\mathbb{Z}} e^{-\pi\alpha yn^2}\vartheta(\frac{y}{\alpha};nx)
-\pi\sqrt{\alpha y}\sum_{n\in\mathbb{Z}} n^2e^{-\pi\alpha yn^2}\vartheta(\frac{y}{\alpha};nx)\\
&+\frac{1}{\alpha}\sqrt{\frac{y}{\alpha}}\sum_{n\in\mathbb{Z}} e^{-\pi\alpha yn^2}\vartheta_X(\frac{y}{\alpha};nx).
 \endaligned\end{equation}
 In the expression of $\theta_y(\alpha,z)$ given by \eqref{P0}, we have
%%%%%%%%%%%%%%%%%%%%%%%%%%%%%%%%%%%%%%%%%%%%%%%

 \begin{equation}\aligned\label{P1}
\sum_{n\in\mathbb{Z}} n^2e^{-\pi\alpha yn^2}\vartheta(\frac{y}{\alpha};nx)\geq0
 \endaligned\end{equation}
and
 \begin{equation}\aligned\label{P2}
\sum_{n\in\mathbb{Z}} e^{-\pi\alpha yn^2}\vartheta_X(\frac{y}{\alpha};nx)
\leq0.
 \endaligned\end{equation}
 Note that $\vartheta(X;Y)=1+2\sum_{n=1}^\infty e^{-\pi n^2X}\cos(2n\pi Y)$.
 By the Poisson summation formula, one has
  \begin{equation}\aligned\label{P222}
\vartheta(X;Y)=X^{-\frac{1}{2}}\sum_{n\in\mathbb{Z}} e^{-\frac{\pi(n-Y)^2}{X}}.
 \endaligned\end{equation}
 It follows by \eqref{P222} that
  \begin{equation}\aligned\label{Positive}
\vartheta(X;Y)\geq0\;\;\hbox{for}\;\; X>0, Y\in\mathbb{R}.
 \endaligned\end{equation}
This proves \eqref{P1}. To prove \eqref{P2}, we first notice that $|\vartheta_X(X;Y)|\leq2\pi\sum_{n=1}^\infty n^2e^{-\pi n^2X}=-\vartheta_X(X;0)$. Splitting the summation to $n=0$ and $n\neq0$, one gets that

\begin{equation}\aligned\nonumber
-\sum_{n\in\mathbb{Z}} e^{-\pi\alpha yn^2}\vartheta_X(\frac{y}{\alpha};nx)
&=-\vartheta_X(\frac{y}{\alpha};0)-2\sum_{n=1}^\infty e^{-\pi\alpha yn^2}\vartheta_X(\frac{y}{\alpha};nx)\\
&\geq-\vartheta_X(\frac{y}{\alpha};0)-2\sum_{n=1}^\infty e^{-\pi\alpha yn^2}\big(-\vartheta_X(\frac{y}{\alpha};0)\big)\\
&=-\vartheta_X(\frac{y}{\alpha};0)
\Big(
1-2\sum_{n=1}^\infty e^{-\pi\alpha yn^2}
\Big)\\
&>0.
 \endaligned\end{equation}

%%%%%%%%%%%%%%%%%%%%%%%%%%%%%%%%%%%%%%%%%%%%%%%
Combining \eqref{P0} with \eqref{P1} and with \eqref{P2}, one gets
 \begin{equation}\aligned\nonumber
\theta_y(\alpha,z)\leq&\frac{1}{2\sqrt{\alpha y}}\sum_{n\in\mathbb{Z}} e^{-\pi\alpha yn^2}\vartheta(\frac{y}{\alpha};nx)\\
=&\frac{1}{2y}\sqrt{\frac{y}{\alpha}}\sum_{n\in\mathbb{Z}} e^{-\pi\alpha yn^2}\vartheta(\frac{y}{\alpha};nx)\\
=&\frac{1}{2y}\theta(\alpha,z).
 \endaligned\end{equation}
This proves item (1) of Lemma \ref{Lemma4.19}.

It remains to prove item (2) of Lemma \ref{Lemma4.19}.
Splitting the summation by $n=0$ and $n\neq0$, one has
 \begin{equation}\aligned\nonumber
\theta(\alpha,y)=\vartheta_3(\frac{\alpha}{y})+2\sqrt\frac{y}{\alpha}\sum_{n=1}^\infty e^{-\pi\alpha yn^2}\vartheta(\frac{y}{\alpha};nx).
 \endaligned\end{equation}

Using \eqref{Positive}, one has
 \begin{equation}\aligned\label{Theta0}
\theta(\alpha,y)\geq\vartheta_3(\frac{\alpha}{y})
\geq 1+2e^{-\pi\frac{\alpha}{y}}.
 \endaligned\end{equation}

While for $\theta_y(\alpha,y)$, we have
 \begin{equation}\aligned\label{C0}
\theta_y(\alpha,y)&=\partial_y(\vartheta_3(\frac{\alpha}{y}))+2\partial_y\Big(\sqrt\frac{y}{\alpha}\sum_{n=1}^\infty e^{-\pi\alpha yn^2}\vartheta(\frac{y}{\alpha};nx)\Big).
 \endaligned\end{equation}
To deal with the second term in $\theta_y(\alpha,y)$, after a direct calculation and regrouping the terms, one gets
 \begin{equation}\aligned\nonumber
\partial_y\Big(\sqrt\frac{y}{\alpha}\sum_{n=1}^\infty e^{-\pi\alpha yn^2}\vartheta(\frac{y}{\alpha};nx)\Big)
=&-\frac{1}{2\sqrt{\alpha y}}\sum_{n=1}^\infty (2\pi\alpha yn^2-1)e^{-\pi\alpha yn^2}\vartheta(\frac{y}{\alpha};nx)\\
&+\frac{1}{\alpha}\sqrt\frac{y}{\alpha}\sum_{n=1}^\infty e^{-\pi\alpha yn^2}\vartheta_X(\frac{y}{\alpha};nx).
 \endaligned\end{equation}

 By \eqref{Positive}, we have

  \begin{equation}\aligned\label{C1}
\partial_y\Big(\sqrt\frac{y}{\alpha}\sum_{n=1}^\infty e^{-\pi\alpha yn^2}\vartheta(\frac{y}{\alpha};nx)\Big)
\leq\frac{1}{\alpha}\sqrt\frac{y}{\alpha}\sum_{n=1}^\infty e^{-\pi\alpha yn^2}\vartheta_X(\frac{y}{\alpha};nx).
 \endaligned\end{equation}
 On the other hand, using the Jacobi theta function,

  \begin{equation}\aligned\label{C2}
\left|\frac{1}{\alpha}\sqrt\frac{y}{\alpha}\sum_{n=1}^\infty e^{-\pi\alpha yn^2}\vartheta_X(\frac{y}{\alpha};nx)\right|
\leq\left|\frac{1}{\alpha}\sqrt\frac{y}{\alpha}\sum_{n=1}^\infty e^{-\pi\alpha yn^2}\vartheta_X(\frac{y}{\alpha};0)\right|
=-\sqrt\frac{y}{\alpha}\sum_{n=1}^\infty e^{-\pi\alpha yn^2}\partial_y(\vartheta_3(\frac{y}{\alpha})).
\endaligned\end{equation}
Via the transformation formula, $\sqrt{\frac{y}{\alpha}}\vartheta_3(\frac{y}{\alpha})=\vartheta_3(\frac{\alpha}{y})$,
by taking derivative with respect to $y$, one has
\begin{equation}\aligned\label{C3}
-\sqrt{\frac{y}{\alpha}}\partial_y(\vartheta_3(\frac{y}{\alpha}))
=\frac{1}{2y}\vartheta_3(\frac{\alpha}{y})-\partial_y(\vartheta_3(\frac{\alpha}{y})).
\endaligned\end{equation}
Since $\partial_y(\vartheta_3(\frac{\alpha}{y}))=\frac{\pi\alpha}{y^2}\sum_{n\in\mathbb{Z}}n^2 e^{-\pi n^2\frac{\alpha}{y}}\geq0$, then
\begin{equation}\aligned\label{C4}
-\sqrt{\frac{y}{\alpha}}\partial_y(\vartheta_3(\frac{y}{\alpha}))
\leq\frac{1}{2y}\vartheta_3(\frac{\alpha}{y}).
\endaligned\end{equation}
Combining \eqref{C0} with \eqref{C1}-\eqref{C4}, we have

 \begin{equation}\aligned\label{Cx}
\theta_y(\alpha,y)\leq\partial_y(\vartheta_3(\frac{\alpha}{y}))+\frac{1}{ y}\vartheta_3(\frac{\alpha}{y})\sum_{n=1}^\infty e^{-\pi\alpha n^2 y}.
 \endaligned\end{equation}
By \eqref{Theta0} and \eqref{Cx}, one gets
 \begin{equation}\aligned\label{Cxx}
\frac{\theta_y(\alpha,z)}{\theta(\alpha,z)}
\leq\frac{\partial_y(\vartheta_3(\frac{\alpha}{y}))+\frac{1}{ y}\vartheta_3(\frac{\alpha}{y})\sum_{n=1}^\infty e^{-\pi\alpha n^2 y}}{1+2e^{-\pi\frac{\alpha}{y}}}
:=\frac{A_1+A_2}{B_1+B_2},
 \endaligned\end{equation}
where
 \begin{equation}\aligned\nonumber
&A_1:=\frac{2\pi\alpha}{y^2}e^{-\frac{\pi\alpha}{y}}, \;\;A_2:=\sum_{n=2}^\infty \frac{2n^2\pi\alpha}{y^2}e^{-\pi n^2\frac{\alpha}{y}}
+\sum_{n=1}^\infty\frac{1}{ y}\vartheta_3(\frac{\alpha}{y})e^{-\pi n^2\alpha y},\\
&B_1:=1,\;\;\;\;\;\;\;\;\;\;\;\;\;\;\; B_2:=2e^{-\pi\frac{\alpha}{y}}.
 \endaligned\end{equation}
By direct computation, one has
 \begin{equation}\aligned\label{Cxxx0}
\frac{A_2}{B_2}<\frac{A_1}{B_1}.
 \endaligned\end{equation}
Then by an elementary inequality, it follows by \eqref{Cxxx0} that
 \begin{equation}\aligned\label{Cxxx}
\frac{A_1+A_2}{B_1+B_2}<\frac{A_1}{B_1}.
 \endaligned\end{equation}
Noting that $\frac{A_1}{B_1}=\frac{2\pi\alpha}{y^2}e^{-\frac{\pi\alpha}{y}}$, \eqref{Cxx} and \eqref{Cxxx} yield the desired result.
\end{proof}

\subsection{Proof of Lemma \ref{Lemma3a}}

Since
\begin{equation}\aligned\nonumber
\frac{\zeta(s,z)}{\theta^k(\alpha,z)}
\Big/\Big(\frac{\zeta(s,z)}{\theta^k(\alpha,z)}\Big)\Big|_{z=e^{i\frac{\pi}{3}}}
=
\frac{\zeta(s,z)}{\zeta(s,e^{i\frac{\pi}{3}})}\cdot\Big(\frac{\theta(\alpha,e^{i\frac{\pi}{3}})}{\theta(\alpha,z)}\Big)^{k},
 \endaligned\end{equation}
we shall estimate $\frac{\zeta(s,z)}{\zeta(s,e^{i\frac{\pi}{3}})}$ and $\frac{\theta(\alpha,e^{i\frac{\pi}{3}})}{\theta(\alpha,z)}$
respectively.

 To estimate $\frac{\theta(\alpha,e^{i\frac{\pi}{3}})}{\theta(\alpha,z)}$, we already have lower bounds of $\theta(\alpha,z)$
in Lemma \ref{Lemma4.3}.
Using the Jacobi Theta function of the third type, we have an upper bound of $\theta(\alpha,z)$, i.e.,

\begin{lemma}\label{Lemma4.21}
For any $\alpha>0, y>0$, it holds that $\theta(\alpha,z)\leq\theta(\alpha,iy)=\vartheta_3(\frac{\alpha}{y})\vartheta_3(\alpha y)$.

\end{lemma}
\begin{proof} By \eqref{CM300}, $\theta(\alpha,z)=\sqrt{\frac{y}{\alpha}}\cdot\sum_{n\in\mathbb{Z}} e^{-\pi\alpha yn^2}\vartheta(\frac{y}{\alpha};nx)\leq\sqrt{\frac{y}{\alpha}}\cdot\sum_{n\in\mathbb{Z}} e^{-\pi\alpha yn^2}\vartheta(\frac{y}{\alpha};0)=\sqrt{\frac{y}{\alpha}}\vartheta_3(\frac{y}{\alpha})\cdot\sum_{n\in\mathbb{Z}} e^{-\pi\alpha yn^2}=\vartheta_3(\frac{\alpha}{y})\vartheta_3(\alpha y)$.
\end{proof}

By Lemmas \ref{Lemma4.3} and \ref{Lemma4.21}, one has

\begin{lemma}\label{Lemma4a} Assume that $\alpha, y>0$. It holds that $\frac{\theta(\alpha,e^{i\frac{\pi}{3}})}{\theta(\alpha,z)}\geq\frac{1}{\theta(\alpha,iy)}=
     \frac{1}{\vartheta_3(\frac{\alpha}{y})\vartheta_3(\alpha y)}$.

\end{lemma}

For $\frac{\zeta(s,z)}{\zeta(s,e^{i\frac{\pi}{3}})}$. Due to its difficulty, we divide it into two cases, i.e., (1): $s\in(1,4]$ and
(2): $s>4$. For case (1), we use Chowla-Selberg formula given by Lemma \ref{Lemma4.14}, while for case (2), we use
summation formula for $\zeta(s,z)$ given by Lemma \ref{Rankin3}.

 We first have
\begin{lemma}\label{Lemma4b}
For $s\in(1,4]$ and $y\geq2$, it holds that
$
\frac{\zeta(s,z)}{\zeta(s,e^{i\frac{\pi}{3}})}\geq\frac{a_0(s,y)-2a_1(s,y)}{a_0(s,\frac{\sqrt3}{2})}.
$ Or equivalently,

 \begin{equation}\aligned\nonumber
\frac{\zeta(s,z)}{\zeta(s,e^{i\frac{\pi}{3}})}\geq
\mathcal{B}_a(s,y),
\endaligned\end{equation}

where

 \begin{equation}\aligned\nonumber
\mathcal{B}_a(s,y):=
\frac{2\xi(2s)y^s+
2\xi(2s-1)\frac{\Gamma(\frac{1}{2})\Gamma(s-\frac{1}{2})}{\Gamma(s)}y^{1-s}
-\frac{8\pi^s\sqrt{y}}{\Gamma(s)}K_{s-\frac{1}{2}}(2\pi y)}
{2\xi(2s)(\frac{\sqrt3}{2})^s+
2\xi(2s-1)\frac{\Gamma(\frac{1}{2})\Gamma(s-\frac{1}{2})}{\Gamma(s)}(\frac{\sqrt3}{2})^{1-s}
}.
\endaligned\end{equation}

\end{lemma}

\begin{proof} We first use Rankin's Lemma (or Montegomery's Lemma \cite{Mon1988,Ran1953}), i.e., $\frac{\partial}{\partial x}\zeta(s,z)<0$ for
$\{\Im(z)\geq\frac{\sqrt3}{2}\}\cap z\in\mathcal{D}_{\mathcal{G}}$ and $s>1$. Then $\zeta(s,z)\geq \zeta(s,\frac{1}{2}+iy)$.
We then use Chowla-Selberg formula given by Lemma \ref{Lemma4.14}.
Indeed,
 \begin{equation}\aligned\label{Zz}
\zeta(s,\frac{1}{2}+iy)=a_0(s,y)+2\sum_{n=1}^\infty (-1)^n a_n(s,y).
\endaligned\end{equation}
See $a_0, a_n$ in Lemma \ref{Lemma4.14}. Using Lemma \ref{Lemma4.15}, one has

 \begin{equation}\aligned\nonumber
\frac{a_{n+1}(y)}{a_n(y)}&=(1+\frac{1}{n})^{s-\frac{1}{2}}\frac{\sigma_{1-2s}(n+1)}{\sigma_{1-2s}(n)}
\frac{K_{s-\frac{1}{2}}(2\pi(n+1)y)}{K_{s-\frac{1}{2}}(2\pi ny)}\\
&\leq(1+\frac{1}{n})^{s-1}\frac{\sigma_{1-2s}(n+1)}{\sigma_{1-2s}(n)} e^{-2\pi y}.
\endaligned\end{equation}
In the range $s\in(1,4]$ and $y\geq\frac{\sqrt3}{2}$, then $\frac{a_{n+1}(y)}{a_n(y)}<1$ for all $n\geq1$.
Hence $\sum_{n=1}^\infty (-1)^n a_n(s,\frac{\sqrt3}{2})$ is an alternating series. Then $\zeta(s,z)\geq a_0(s,y)- 2a_1(s,y)$.
For $\zeta(s,e^{i\frac{\pi}{3}})$,  still using \eqref{Zz} and the series is alternating, then $\zeta(s,e^{i\frac{\pi}{3}})\leq a_0(s,y)$.
These yield the result.

\end{proof}

The lower bound function $\mathcal{B}_a(s,y)$ is monotone on $s$ and $y$ directions. It follows that
\begin{lemma}\label{Lemma4c}It holds that $\min_{s\in[1,4], y\geq2}\mathcal{B}_a(s,y)=\mathcal{B}_a(1,2)$. Here
$\mathcal{B}_a(1,2)=1.133290376\cdots$.

\end{lemma}
Now we are ready to prove the case $s\in(1,4]$ of Lemma \ref{Lemma3a}. Namely,
\begin{lemma}\label{L31}
For $s\in(1,4]$, $y\geq2, \alpha\geq3s$ and $k\leq2s$, it holds that
\begin{equation}\aligned\nonumber
\frac{\zeta(s,z)}{\zeta(s,e^{i\frac{\pi}{3}})}\cdot\Big(\frac{\theta(\alpha,e^{i\frac{\pi}{3}})}{\theta(\alpha,z)}\Big)^{k}>1.
 \endaligned\end{equation}

\end{lemma}

\begin{proof} By Lemmas \ref{Lemma4a}, \ref{Lemma4b} and \ref{Lemma4c},
\begin{equation}\aligned\nonumber
\frac{\zeta(s,z)}{\zeta(s,e^{i\frac{\pi}{3}})}\cdot\Big(\frac{\theta(\alpha,e^{i\frac{\pi}{3}})}{\theta(\alpha,z)}\Big)^{k}
\geq &\mathcal{B}_a(s,y)\Big(\frac{\theta(\alpha,e^{i\frac{\pi}{3}})}{\theta(\alpha,z)}\Big)^{2s}
\geq \mathcal{B}_a(s,y)\Big(\frac{1}{\theta(3s,iy)}\Big)^{2s}\\
\geq &\mathcal{B}_a(1,2)\Big(\frac{1}{\theta(3,2i)}\Big)^{2}=1.093639371\cdots
>1.
 \endaligned\end{equation}

\end{proof}

For large $s$, the Chowla-Selberg formula given by Lemma \ref{Lemma4.14} does not work  well due to the property of $K_s(y)$. Instead, we use
summation formula by Lemma \ref{Rankin3}. A direct consequence of Lemma \ref{Rankin3} gives that
\begin{lemma}\label{Lemma4d}
For $s\geq4$ and $y\geq2$, it holds that

 \begin{equation}\aligned\nonumber
\frac{\zeta(s,z)}{\zeta(s,e^{i\frac{\pi}{3}})}\geq
\mathcal{B}_b(s,y).
\endaligned\end{equation}

Here

 \begin{equation}\aligned\nonumber
\mathcal{B}_b(s,y):=
\frac{2\xi(2s)y^s+
2\xi(2s-1)\frac{\Gamma(\frac{1}{2})\Gamma(s-\frac{1}{2})}{\Gamma(s)}y^{1-s}
-\xi(2s+1)s\frac{(2s+1)^{s+\frac{1}{2}}}{(2s+2)^{s+1}}y^{-(s+1)}}
{2\xi(2s)(\frac{\sqrt3}{2})^s+
2\xi(2s-1)\frac{\Gamma(\frac{1}{2})\Gamma(s-\frac{1}{2})}{\Gamma(s)}(\frac{\sqrt3}{2})^{1-s}
+\xi(2s+1)s\frac{(2s+1)^{s+\frac{1}{2}}}{(2s+2)^{s+1}}(\frac{\sqrt3}{2})^{-(s+1)}}.
\endaligned\end{equation}

\end{lemma}

Now we are ready to prove the case $s>4$ of Lemma \ref{Lemma3a}. Namely,
\begin{lemma}\label{L32}
For $s>4$, $y\geq2, \alpha\geq3s$ and $k\leq2s$, it holds that
\begin{equation}\aligned\nonumber
\frac{\zeta(s,z)}{\zeta(s,e^{i\frac{\pi}{3}})}\cdot\Big(\frac{\theta(\alpha,e^{i\frac{\pi}{3}})}{\theta(\alpha,z)}\Big)^{k}>1.
 \endaligned\end{equation}

\end{lemma}

\begin{proof} For $s>4$ and $y\geq2$, one trivially has $\mathcal{B}_b(s,y)>\frac{y^s}{s(\frac{\sqrt3}{2})^{s+1}}$.
 By Lemmas \ref{Lemma4a} and \ref{Lemma4d},
\begin{equation}\aligned\nonumber
\frac{\zeta(s,z)}{\zeta(s,e^{i\frac{\pi}{3}})}\cdot\Big(\frac{\theta(\alpha,e^{i\frac{\pi}{3}})}{\theta(\alpha,z)}\Big)^{k}
\geq \mathcal{B}_b(s,y)\Big(\frac{1}{\theta(3s,iy)}\Big)^{2s}
\geq \frac{y^s}{s(\frac{\sqrt3}{2})^{s+1}}\Big(\frac{1}{\theta(12,2i)}\Big)^{8}>\sqrt3.
 \endaligned\end{equation}

\end{proof}

By Lemmas \ref{L31} and \ref{L32}, we complete the proof of Lemma \ref{Lemma3a}.

\subsection{Proof of Corollary \ref{Coro1.3}}

By a deformation
 \begin{equation}\aligned\nonumber
{\zeta(s,z)}-{\theta^k(\alpha,z)}=\theta^k(\alpha,z)\cdot\big(
\frac{\zeta(s,z)}{\theta^k(\alpha,z)}-1
\big),
\endaligned\end{equation}
and Theorem \ref{Th2} and Montgomery's Theorem B (\cite{Mon1988}), to prove the first part of Corollary \ref{Coro1.3},  it suffices to prove that

\begin{lemma}\label{Lemma4.28}
Assume that $s\in(1,12],\;\; \alpha\geq 3s$. Then for $k\in(0,2s]$, it holds that
$
\frac{\zeta(s,e^{i\frac{\pi}{3}})}{\theta^k(\alpha,e^{i\frac{\pi}{3}})}\geq1.
$

\end{lemma}

Since $\theta(\alpha,z)\geq1$ by Lemma \ref{Lemma4.3}, to prove Lemma \ref{Lemma4.21}, it suffices to prove that

\begin{lemma}\label{Lemma4.22}
Assume that $s\in(1,12]$. Then we have
$
\frac{\zeta(s,e^{i\frac{\pi}{3}})}{(\vartheta_3(2\sqrt3 s)\vartheta(\frac{3\sqrt3}{2} s))^{2s}}\geq1.
$

\end{lemma}
Note that $(\vartheta_3(2\sqrt3 s)\vartheta(\frac{3\sqrt3}{2} s))$ is very close to 1,
and $(\vartheta_3(2\sqrt3 )\vartheta(\frac{3\sqrt3}{2} ))=1.000112671\cdots$.
To prove Lemma \ref{Lemma4.22}, we use Chowla-Selberg formula given by Lemma \ref{Lemma4.14}.
The proof is similar to that in Lemma \ref{Lemma4b}, we omit the details here.

\bigskip
\noindent
{\bf Acknowledgements.} The research of S. Luo is partially supported by the Jiangxi Jieqing Fund under Grant No. 20242BAB23001, and by the National Natural Science Foundation of China (NSFC) under Grant Nos. 12261045 and 12001253. The research of J. Wei is partially supported by the General Research Fund (GRF) of Hong Kong "New frontiers in singularity formation of nonlinear partial differential equations".

{\bf Statements and Declarations: there is no conflict of interest.}

{\bf Data availability: the manuscript has no associated data.}
\bigskip

%%%%%%%%%%%%%%%%%%%%%%%%%%%%%%%%%%%%%%%%%%%%%%%%%%%%%%%%%%%%%%%%%%%%%%%%%%%%%%%%%%%%%%%%%%%%%%%%%%%%
%%%%%%%%%%%%%%%%%%%%%%%%%%%%%%%%%%%%%%%%%%%%%%%%%%%%%%%%%%%%%%%%%%%%%%%%%%%%%%%%%%%%%%%%%%%%%%%%%%%%


\begin{thebibliography}{10}

\bibitem{Abr} A. Abrikosov, Nobel Lecture: Type-II superconductors and the vortex lattice. {\em Reviews of modern physics} 76(2004), no.3, p. 975.

\bibitem{Vafa1986}
L. Alvarez-Gaum\'e, G. Moore, C. Vafa, Theta functions, modular invariance, and strings. {\em Comm. Math. Phys.} 106 (1986), no. 1, 1-40.

\bibitem{Afhk2021}
N. Afkhami-Jeddi, H. Cohn, T. Hartman, A. Tajdini, Free partition functions and an averaged holographic
duality. {\it JHEP} 01, 130 (2021).

\bibitem{Bari2010}
\'A. Baricz, Bounds for modified Bessel functions of the first and second kinds.
{\em Proc. Edinb. Math. Soc.} (2) 53 (2010), no. 3, 575-599.

\bibitem{Bar2020}
A. Barreal, M. Damir, R. Freij-Hollanti, C. Hollanti,
An approximation of theta functions with applications to communications.
{\em SIAM J. Appl. Algebra Geom.} 4 (2020), no. 4, 471-501.


\bibitem{Ben2022}
N. Benjamin, C. Keller, H. Ooguri, I. Zadeh, Narain to Narnia. {\em Comm. Math. Phys.} 390 (2022), no. 1, 425-470.


\bibitem{Bet2016}  L. B\'etermin,
Two-dimensional theta functions and crystallization among Bravais lattices, {\em SIAM Journal on Mathematical Analysis},
 48(5) (2016), 3236-269.


\bibitem{Bet2018}  L. B\'etermin,
Local variational study of 2d lattice energies and application to Lennard-Jones type interactions, {\em Nonlinearity}
31(9) (2018), 3973-4005.



\bibitem{Bet2019}  L. B\'etermin,
Minimizing lattice structures for Morse potential energy in two and three dimensions, {\em  J. Math. Phys.},
60(10) (2019), 102901.

\bibitem{Bet2019AMP} L. B\'etermin, M. Petrache,
Optimal and non-optimal lattices for non-completely monotone interaction potentials, {\em Anal. Math. Phys.} 9(4): 2033-2073, 2019.

%\bibitem{BP} L. B\'etermin and M. Petrache,  Dimension reduction techniques for the minimization of theta functions on lattices. {\em Journal of Mathematical Physics}, 58(7)(2017), 071902.

%\bibitem{Bet2020} L. B\'etermin, M. Faulhuber and H. Kn$\ddot{u}$pfer,
%On the optimality of the rock-salt structure among lattices with charge distributions, {\em Mathematical Models and Methods in Applied Sciences} 31(2):293-325, 2021.


\bibitem{Betermin2021JPA}  L. B\'etermin, On energy ground states among crystal lattice structures with prescribed bonds,
 {\em J. Phys. A} 54 (2021), no. 24, Paper No. 245202, 18 pp.


\bibitem{Betermin2021AHP}
 L. B\'etermin, Effect of periodic arrays of defects on lattice energy minimizers. {\em Ann. Henri Poincar\'e} 22 (2021), no. 9, 2995-3023.

\bibitem{Betermin2021a}  L. B\'etermin,
Optimality of the triangular lattice for Lennard-Jones type lattice
energies: a computer-assisted method, {\em J. Phys. A} 56 (2023), no. 14, Paper No. 145204, 19 pp.


\bibitem{Betermin2020CA}
  L. B\'etermin, Minimal soft lattice theta functions. {\em Constr. Approx.} 52 (2020), no. 1, 115-138.


\bibitem{Betermin2023JAM}
 L. B\'etermin, M. Faulhuber, Maximal theta functions universal optimality of the hexagonal lattice for Madelung-like lattice energies. {\em J. Anal. Math.} 149 (2023), no. 1, 307-341.

\bibitem{Betermin2021SIAM}
 L. B\'etermin, Theta functions and optimal lattices for a grid cells model. {\em SIAM J. Appl. Math.} 81 (2021), no. 5, 1931-1953.

\bibitem{Betermin2021bb} L. B\'etermin, M. Faulhuber, and S. Steinerberger,
A variational principle for Gaussian lattice sums, arXiv:2110.006008v1.

\bibitem{Bers1991}
M. Bershadsky, I. Klebanov, Partition functions and physical states in two-dimensional quantum gravity and supergravity.
 {\em Nuclear Phys. B} 360 (1991), no. 2-3, 559-585.

 \bibitem{Bla2015}
X. Blanc, M. Lewin, The Crystallization Conjecture: A Review. {\em EMS Surveys in Mathematical Sciences, EMS}  2(2)2015, 255-306.


\bibitem{Che2007}
X. Chen, Y. Oshita, An application of the modular function in nonlocal variational problems. {\em Arch.
Rat. Mech. Anal.}, 186(1) (2007), 109-132.

\bibitem{Cas1959} J. Cassels,  On a problem of Rankin about the Epstein Zeta function, {\em Proc. Glasgow
Math. Assoc.} 4(1959), 73-80. (Corrigendum, ibid. 6 (1963), 116.)

\bibitem{Cohn2018}
H. Cohn, M. de Courcy-Ireland,
The Gaussian core model in high dimensions.
{\em Duke Math. J.} 167 (2018), no. 13, 2417-2455.


\bibitem{Coh2017}
H. Cohn, A. Kumar, S. Miller, D. Radchenko, M. Viazovska, The sphere packing problem in dimension 24, {\em Annals of Mathematics} 2017, 1017-1033.

\bibitem{Conway}
J. Conway, N. Sloane, Sphere packings, lattices and groups. With contributions by E. Bannai, J. Leech, S. Norton, A. Odlyzko, R. Parker, L. Queen and B. Venkov. Grundlehren der mathematischen Wissenschaften [Fundamental Principles of Mathematical Sciences], 290. Springer-Verlag, New York, 1988. xxviii+663 pp. ISBN: 0-387-96617-X.

\bibitem{Chowla1967}
S. Chowla, A. Selberg, On Epstein's Zeta function, {\em J. Reine Angew. Math.} 227
(1967), 86-110.


\bibitem{Chowla1949}
A. Selberg, S. Chowla, On Epstein's Zeta function (I), {\em Proc. Nat. Acad.
Sci.} 35 (1949), 371-74.


\bibitem{Cohen2007}
H. Cohen, Number theory. Vol. II. Analytic and modern tools. Graduate Texts in Mathematics, 240. Springer, New York, 2007. xxiv$+$596 pp. ISBN: 978-0-387-49893-5.

\bibitem{Dia1964} P. Diananda, Notes on two lemmas concerning the Epstein zeta-function, {\em Proc. Glasgow
Math. Assoc.} 6 (1964), 202-204.


\bibitem{Enn1964a} V. Ennola, A lemma about the Epstein Zeta function, {\em Proc. Glasgow Math. Assoc.} 6
(1964), 198-201.


\bibitem{Folkins1991}
I. Folkins, Functions of two-dimensional Bravais lattices. {\em J. Math. Phys.}
32, 1965-1969 (1991).


\bibitem{GMS2013} D. Goldman, C. Muratov, S. Serfaty, The Gamma-limit of the two-dimensional Ohta-Kawasaki
energy. I. droplet density. {\em  Arch. Rat. Mech. Anal.} 210(2)(2013), 581-613.

\bibitem{Ho2001} T. Ho, Bose-Einstein condensates with large number of vortices. {\em Physical Review Letters} 87(2001), 604031-604034

\bibitem{Kramer1991}
J. Kramer,
A geometrical approach to the theory of Jacobi forms.
{\em Compositio Math.} 79 (1991), no. 1, 1-19.


\bibitem{L}
S. Lefschetz, On the Functional Independence of Ratios of Theta Functions, {\it Proceedings of the National Academy of Sciences of the United States of America}, Vol. 13, No. 9 (Sep. 15, 1927), pp. 657-659 (3 pages).

\bibitem{Luo2019} S. Luo, X. Ren, J. Wei,
Non-hexagonal lattices from a two species interacting system, {\em SIAM J. Math. Anal.}, 52(2) (2020), 1903-1942.

\bibitem{Luo2023a}
 S. Luo, J. Wei, On lattice hexagonal crystallization for non-monotone potentials, {\em J. Math. Phys.} 65 (2024), no. 7, Paper No. 071901, 28 pp.

\bibitem{LW2022}
  S. Luo, J. Wei, On minima of sum of theta functions and application to Mueller-Ho conjecture. {\em Arch. Ration. Mech. Anal.} 243 (2022), no. 1, 139-199.


\bibitem{LW2023}
  S. Luo, J. Wei, On minima of difference of theta functions and application to hexagonal crystallization, {\em Math. Ann.} 387 (2023), no. 1-2, 499-539.

\bibitem{LW2025JMP}
  S. Luo, J. Wei, On non-Gaussian potentials having triangular lattice as minimizer at any fixed density, {\em J. Math. Phys.}, 66 (2025), no. 12, Paper No. 121901, 33 pp.

\bibitem{LW2025}
  S. Luo, J. Wei,
On minima of difference of Epstein zeta functions and exact solutions to Lennard-Jones lattice energy,  {\em J. Eur. Math. Soc. $($JEMS$)$}, (2025), published online first, DOI 10.4171/JEMS/1682.


\bibitem{Pin2020}
I. Pinelis, Exact lower and upper bounds on the incomplete gamma function. {\em Math. Inequal. Appl.} 23 (2020), no. 4, 1261-1278.

\bibitem{Witten2020} A. Maloney, E. Witten, Averaging over Narain moduli space. {\it JHEP} 10, 187 (2020).


\bibitem{Mon1988}
H. Montgomery, Minimal theta functions. {\em Glasgow Math. J.} 30 (1988), 75-85.

\bibitem{Mumford1983}
D. Mumford, Tata lectures on theta. I. With the assistance of C. Musili, M. Nori, E. Previato and M. Stillman. Progress in Mathematics, 28. Birkh$\ddot{a}$user Boston, Inc., Boston, MA, 1983. xiii+235 pp. ISBN: 3-7643-3109-7.


\bibitem{Naka2004} Y. Nakayama, Liouville field theory: a decade after the revolution, {\em International Journal of Modern Physics A} Vol. 19, No. 17-18, pp. 2771-2930 (2004).

\bibitem{Non1998}
 S. Nonnenmacher and A. Voros, Chaotic eigenfunctions in phase space, {\em J. Statist. Phys.}, 92 (1998), 431-518.

\bibitem{Sarnak1988} B. Osgood, R. Phillips, P. Sarnak,
Extremals of determinants of Laplacians, {\em Journal of functional analysis} 80, 148-211 (1988).

\bibitem{Pres2011} S. Prestipino, F. Saija, P. Giaquinta,
Hexatic Phase in the Two-Dimensional Gaussian-Core Model,
{\em Phys. Rev. Lett.} 106, 235701 - Published 10 June 2011.


\bibitem{Regev2024}
 O. Regev, N. Stephens-Davidowitz, A reverse Minkowski theorem. {\em Ann. of Math.} (2) 199 (2024), no. 1, 1-49.

\bibitem{Regev2023}
O. Regev, Some questions related to the reverse Minkowski theorem. ICM-International Congress of Mathematicians. Vol. 6. Sections 12-14, 4898-4912, EMS Press, Berlin, 2023.

\bibitem{Serfaty2012}
E. Sandier, S. Serfaty, From the Ginzburg-Landau model to vortex lattice problems. {\em Comm. Math. Phys.} 313(2012), 635-743.

\bibitem{Sarnak2006}
P. Sarnak, A. Str\"{o}mbergsson, Minima of Epstein's zeta function and heights of flat tori. {\em Invent. Math.} 165 (2006), no. 1, 115-151.

\bibitem{Serfaty2018}
S. Serfaty, Systems of points with Coulomb interactions. {\em Proceedings of the International Congress of Mathematicians-Rio de Janeiro 2018.} Vol. I. Plenary lectures, 935-977, World Sci. Publ., Hackensack, NJ, 2018.

\bibitem{Via2017} M. Viazovska, The sphere packing problem in dimension 8, {\em Annals of Mathematics} 2017, 991-1015.

\bibitem{Fran1997}
P. Francesco, P. Mathieu, D. S\'en\'echal, Conformal Field Theory. {\it Graduate Texts in Contemporary Physics}, New York, NY: Springer New York. p. 104. ISBN 978-1-4612-2256-9, 1997.


\bibitem{Muss2020} G. Mussardo,
Statistical Field Theory:
An Introduction to Exactly Solved Models in Statistical Physics,
Second Edition, {\it Oxford Graduate Texts}, ISBN: 9780198788102, 2020.

\bibitem{Ran1953} R. Rankin, A minimum problem for the Epstein Zeta function, {\em Proc. Glasgow Math.
Assoc.} 1 (1953), 149-158.

\bibitem{Watson1995}
G. Watson, A treatise on the theory of Bessel functions. Reprint of the second (1944) edition. Cambridge Mathematical Library. Cambridge University Press, Cambridge, 1995. viii$+$804 pp. ISBN: 0-521-48391-3.

\bibitem{Zagier} D. Zagier,
The Rankin-Selberg method for automorphic functions which are not of rapid decay,
{\em J. Fac. Sci. Tokyo} 28 (1982), 415-438.

\end{thebibliography}
\end{document}